\documentclass{amsart}
\usepackage[foot]{amsaddr}
\usepackage{amssymb}
\usepackage{mathtools}
\usepackage[hidelinks]{hyperref}
\usepackage[numbered]{bookmark}     
\usepackage[maxbibnames=10, 
  url=false, 
  doi=false, 
  date=year,
  isbn=false]{biblatex}
\usepackage{enumitem}   
\usepackage{braket}     
\usepackage{todonotes}  
\usepackage{comment}
\newcommand{\comments}[1]{}
\usepackage{multicol}
\usepackage{array}      
\usepackage{rotating} 
\usepackage{arydshln} 
\usepackage{booktabs} 
\usepackage{array}
\usepackage{tabu} 

\title{Complete Classification of Directed Quantum Graphs on $M_2$}
\author{Nina Kiefer and Bj\"orn Sch\"afer}
\email{kiefer@math.uni-sb.de, bschaefer@math.uni-sb.de}
\address{Department of Mathematics, Saarland University, D-66123 Saarbr\"ucken, Germany}
\date{\today}

\newtheorem{thm}{Theorem}[section]
\newtheorem{definition}[thm]{Definition}
\newtheorem{lemma}[thm]{Lemma}

\newtheorem{proposition}[thm]{Proposition}
\newtheorem*{proposition*}{Proposition}

\newtheorem*{lemma*}{Lemma}
\newtheorem*{definition*}{Definition}

\newtheorem{example}[thm]{Example}
\newtheorem{remark}[thm]{Remark}

\newcommand{\cc}{\vcentcolon}

\newcommand{\Tr}{\mathrm{Tr}}

\newcommand{\R}{\mathbb{R}}
\newcommand{\C}{\mathbb{C}}

\newcommand{\N}{\mathbb{N}}

\DeclareMathOperator{\SO}{SO}

\DeclareMathOperator{\spann}{span}
\DeclareMathOperator{\op}{op}
\DeclareMathOperator{\Id}{Id}

\DeclareMathOperator{\KMS}{KMS}

\AtEveryBibitem{\clearlist{language}} 

\begin{document}

\begin{abstract}
    In 2022,  Gromada and Matsuda classified undirected quantum graphs on the matrix algebra $M_2$ \cite{gromada_examples_2022, matsuda_classification_2022}. Later, Wasilweski provided a solid theory of directed quantum graphs \cite{wasilewski_quantum_2024} which was formerly only established for undirected quantum graphs. Using this framework we extend the results of Matsuda and Gromada, and present a complete classification of directed quantum graphs on $M_2$. Most prominently, we observe that there is a far bigger range of directed quantum graphs than of undirected quantum graphs on $M_2$. Moreover, for quantum graphs on a nontracial quantum set $(M_2, \psi)$ we illustrate the difference between GNS- and KMS-undirected quantum graphs.
\end{abstract}

\maketitle

\setcounter{tocdepth}{1}    
\bookmarksetup{depth=2}     
\tableofcontents

\section{Introduction}

Quantum graphs are a generalization of ordinary graphs where the vertex set $\{1, \dots, n\}$ is replaced with a quantum set $(B, \psi)$ consisting of a finite dimensional $C^\ast$-algebra $B$ and a particular faithful functional $\psi$ on $B$.
Ordinary graphs are retained as the special case where $B = \C^n$ and $\psi(e_i) = \frac{1}{n}$ for $i=1,\dots, n$.
The edge relation of an ordinary graph can be generalized to the quantum graph setting in three different ways: As a $B^\prime$-$B^\prime$-bimodule $S \subset B(\mathcal{H})$ given an embedding $B \subset B(\mathcal{H})$, as a projection $P \in B \otimes B^\mathrm{op}$ or as a particular linear map $A: B \to B$. The latter is usually called a quantum adjacency matrix. 
Three important properties of quantum graphs are 1) being undirected, 2) being tracial, meaning that $\psi$ is a tracial functional, and 3) being reflexive/loopfree\footnote{The reflexive quantum graphs on a fixed quantum set $(B, \psi)$ are in a 1-1-correspondence with loopfree quantum graphs.}.

Let us give a short and selective summary of the history of quantum graphs.
They were first introduced by Duan, Severini and Winter in the context of quantum information theory where they allow generalizing confusability graphs of classical information channels to the quantum world \cite{duan_zero-error_2013}. At about the same time, Weaver proposed a more general framework of quantum relations \cite{weaver_quantum_2012}. Musto, Reutter and Verdon, finally, introduced the notion of a quantum adjacency matrix \cite{musto_compositional_2018}, and proved that the three ways to define quantum graphs are equivalent when $\psi$ is tracial and the quantum graph is undirected. Extending this, Daws proved the same equivalence in the more general situation when $\psi$ is not necessarily tracial \cite{daws_quantum_2024}. With the consistent use of the KMS-inner product, Wasilewski at last covered also the case of not necessarily undirected quantum graphs \cite{wasilewski_quantum_2024}. We call such quantum graphs \emph{directed}. Note that, therefore, undirected quantum graphs are a special case of directed quantum graphs.

The smallest non-commutative finite-dimensional C*-algebra is the complex matrix algebra $M_2$. In order to obtain a better understanding of quantum graphs it is therefore a natural starting point to try and describe all quantum graphs on quantum sets of the form $(M_2, \psi)$. This has been done for undirected quantum graphs by Gromada and Matsuda \cite{gromada_examples_2022, matsuda_classification_2022}.

The present paper is concerned with the complete classification of directed quantum graphs on $M_2$, i.e. we solve the following problem:
\begin{center}
    \textit{Describe all directed quantum graphs on $(M_2, \psi)$ up to isomorphism.}
\end{center}
Thus, we consider all faithful positive functionals $\psi$ on $M_2$ which turn $(M_2, \psi)$ into a quantum set. There is exactly one such functional $\psi = \tau$ that is tracial, and the nontracial admissible functionals $\psi_q$ are indexed by a parameter $q \in (0,1)$. For the purposes of classification, the quantum graphs on $(M_2, \psi_q)$ for different $q$ can be treated at the same time. Therefore, we mainly need to distinguish between the tracial case $(M_2, \tau)$ and the nontracial cases $(M_2, \psi_q)$.

Let us discuss the prior work of Gromada and Matsuda \cite{gromada_examples_2022,matsuda_classification_2022}. Independently of each other, they classified undirected quantum graphs on $(M_2, \psi)$ in different generality. While Matsuda considered reflexive/loopfree quantum graphs for both tracial and nontracial quantum sets $(M_2, \psi)$, Gromada classified non-reflexive quantum graphs on tracial $(M_2, \psi)$. They found that there are up to isomorphism only four reflexive/loopfree quantum graphs on $(M_2, \psi)$ for both tracial and nontracial quantum sets $(M_2, \psi)$. 
Further, Matsuda showed that these quantum graphs are each quantum isomorphic to a classical graph on four vertices. For our classification, we use Gromada's approach, classifying quantum graphs on $(M_2, \psi)$ via their quantum edge space $S \subset M_2$.


In what follows, we would like to highlight some special cases of our classification which might be particularly insightful. The difference between the classification of undirected and directed quantum graphs is best seen in the special case of tracial and loopfree quantum graphs on $(M_2, \tau)$, where $\tau$ is the unique tracial functional that turns $M_2$ into a quantum set. Whereas there is exactly one undirected quantum graph with exactly one quantum edge on $(M_2, \tau)$, the corresponding directed quantum graphs are parametrized by a parameter $\beta \in [0,1]$ with $\beta = 0$ giving the previously known undirected quantum graph.

\begin{thm}[Theorem \ref{t1e::thm:classification_of_1te_qgraphs}, Lemma \ref{pre::lemma:1-1-correspondence_between_qgraphs_with_k_and_n2-k(-1)_edges}]
    Let $\mathcal{G}$ be a loopfree quantum graph on the (tracial) quantum set $(M_2, \tau)$.
    \begin{enumerate}
        \item If $\mathcal{G}$ has exactly one quantum edge, then it is isomorphic to exactly one of the quantum graphs $\mathcal{G}_{0,\beta}^{(1B)}$ from Proposition \ref{t1e::prop:properties_of_G_alpha,beta^(1B)} with $\beta \in [0,1]$ given by the quantum adjacency matrix 
        \begin{align*}
            A_{0,\beta}^{(1B)}=
        \frac{1}{1+\beta^2}
        \begin{pmatrix}
            0 & 0 & 0 & (1+\beta)^2 \\
            0 & 0 & 1 - \beta^2 & 0 \\
            0 & 1-\beta^2 & 0 & 0 \\
            (1-\beta)^2 & 0 & 0 & 0
        \end{pmatrix}.
        \end{align*}
        The quantum graph $\mathcal{G}_{0,\beta}^{(1B)}$ is undirected if and only if $\beta = 0$, and has spectrum
        \begin{align*}
            \sigma_\beta =\frac{1}{1+\beta^2} \left\{1-\beta^2, 1-\beta^2, \beta^2-1, \beta^2-1\right\}.
        \end{align*}
        \item If $\mathcal{G}$ has two quantum edges, then it is isomorphic to the loopfree complement\footnote{Loopfree complements of quantum graphs are introduced in Definition \ref{pre::def:(looopfree)_complement_of_qgraphs}.} of one of the quantum graphs from (1).
        \item If $\mathcal{G}$ has three quantum edges, then it is the complete loopfree quantum graph on $(M_2, \tau)$.
    \end{enumerate}
\end{thm}

More generally, we classify all quantum graphs on $(M_2, \tau)$ with Theorems \ref{t1e::thm:classification_of_1te_qgraphs} and \ref{t2e::thm:classification_of_t2e_qgraphs}. Dropping the condition that $\mathcal{G}$ is loopfree introduces further parameters in the quantum adjacency matrix.

We would like to highlight that at the time of Matsuda's and Gromada's classification, the theory of directed quantum graphs was not fully settled. Crucially, in this generality the equivalence of the different definitions of quantum graphs was only proved by Wasilewski in \cite{wasilewski_quantum_2024}. There he also introduced a new definition for undirected quantum graphs on nontracial quantum sets, which amounts to taking the KMS-adjoint of the quantum adjacency matrix instead of the GNS-adjoint, see Section \ref{pre::subsec:properties_of_quantum_graphs} for a discussion. When $\psi$ is tracial these two concepts coincide. Generally, however, they are different, and this is best illustrated by the classification of nontracial loopfree quantum graphs on $M_2$. Again these directed quantum graphs are parametrized by one or more parameters.

\begin{thm}[Theorem \ref{ntg::thm:classification_nontracial_qgraphs_with_one_edge}, Lemma \ref{pre::lemma:1-1-correspondence_between_qgraphs_with_k_and_n2-k(-1)_edges}]
    Let $\mathcal{G}$ be a loopfree quantum graph on $(M_2, \psi_q)$.
    \begin{enumerate}
        \item If $\mathcal{G}$ has exactly one quantum edge, then it is isomorphic to exactly one of the following quantum graphs.
        \begin{enumerate}
            \item The quantum graph $\mathcal{G}_{0}^{(q,1B)}$ from Proposition \ref{ntg::prop:qgraph_G_alpha_on_(M2,psi_q)_with_one_edge} given by the adjacency matrix 
            \begin{align*}
                A_{0}^{(q,1B)} = \begin{pmatrix}
                    q^{-2} & 0 & 0 & 0 \\
                    0 & -1 & 0 & 0 \\
                    0 & 0 & -1 & 0 \\
                    0 & 0 & 0 & q^{2}
                \end{pmatrix}.
            \end{align*}
            This quantum graph is both GNS- and KMS-undirected.
            \item The quantum graph $\mathcal{G}_{0, \beta, \gamma}^{(q,1C)}$ from Proposition \ref{ntg::prop:qgraph_G_alpha,beta,gamma_with_one_edge} with $\beta \in [-1,1]$ and $\gamma \in \C$ such that  
            \begin{align*}
                |\beta| < 1 & \implies   \gamma \geq 0, \\
                |\beta| = 1 & \implies   \mathrm{arg}(\gamma) \in [0, \pi),
            \end{align*}
            given by the quantum adjacency matrix
            \begin{align*}
                A_{0, \beta, \gamma}^{(q,1C)} = c^{-1} \begin{pmatrix}
                    q^{-4} |\gamma|^2     & q^{-\frac{3}{2}} \beta_+ \gamma            & q^{-\frac{3}{2}}\beta_+ \overline{\gamma}              & q \beta_+^2            \\
                    q^{-\frac{5}{2}} \beta_- \gamma & - q^{-2} |\gamma|^2 & \beta_+ \beta_-                           & -q^{\frac{1}{2}}\beta_+ \overline{\gamma}   \\
                    q^{\frac{3}{2}}\beta_- \overline{\gamma} & \beta_+ \beta_-                          & - q^{-2} |\gamma|^2   & -q^{\frac{1}{2}}\beta_+ \gamma   \\
                    q^{-1}\beta_-^2         & - q^{-\frac{1}{2}}\beta_- \overline{\gamma}            & -q^{-\frac{1}{2}} \beta_- \gamma              & |\gamma|^2        
                \end{pmatrix},
            \end{align*}
            where $\beta_+ = 1+\beta$, $\beta_- = 1-\beta$ and
            \begin{align*}
                c = \frac{1}{1+q^2} \left( \left(q^{-2}+1\right)|\gamma|^2 + 2q\left(1+\beta^2\right)\right).
            \end{align*}
            This quantum graph is not GNS-undirected, and it is KMS-undirected if and only if $\beta = 0$ (and $\gamma \geq 0$).
        \end{enumerate}
        \item If $\mathcal{G}$ has two quantum edges, then it is isomorphic to the loopfree complement of one of the quantum graphs from (1).
        \item If $\mathcal{G}$ has three quantum edges, then it is the complete loopfree quantum graph on $(M_2, \psi_q)$.
    \end{enumerate}
\end{thm}

The quantum graph $\mathcal{G}_0^{(q,1B)}$ has been described by Matsuda, while the quantum graphs $\mathcal{G}_{0, 0, \gamma}^{(q,1)}$ with $\gamma \geq 0$ yield a multitude of KMS-undirected but not GNS-undirected quantum graphs on $(M_2, \psi_q)$. More generally, we classify all quantum graphs on $(M_2, \psi_q)$ with Theorems \ref{ntg::thm:classification_nontracial_qgraphs_with_one_edge} and \ref{ntg::thm:classification_of_nontracial_qgraphs_with_two_edges}.

\subsection{Outline} 

The rest of this paper is organized as follows: In \textbf{Section \ref{sec::preliminaries}} we recall key concepts such as quantum sets and quantum graphs based mainly on \cite{wasilewski_quantum_2024}. In \textbf{Section \ref{sec::qgraph_isomorphism}} we discuss isomorphisms of quantum graphs. Then, \textbf{Section \ref{sec::classification_of_lines}} contains the main technical preparation for the later work: a classification of lines in $\C^2$ and $\C^3$ up to certain orthogonal transformations. With that in hand, we classify tracial quantum graphs on $M_2$ in \textbf{Sections \ref{sec::t1e}} and \textbf{\ref{sec::t2e}} -- first the quantum graphs with one quantum edge and then the non-loopfree quantum graphs with two quantum edges. Due to Lemma \ref{iso::remark:complete_classification}, this is sufficient to obtain the complete classification for tracial quantum graphs. In \textbf{Sections \ref{sec::n1e}} and \textbf{\ref{sec::n2e}} we repeat these arguments for the more general case of nontracial quantum graphs. Lastly, in \textbf{Section \ref{sec::cc}} we summarize the complete classification and provide a tabular overview.

\subsection{Acknowledgments}

This work originated during the Workshop on Quantum Graphs held at Saarland University in February 2025. The authors would like to thank Daniel Gromada and Luca Junk for organizing this event. The authors are also grateful to their supervisor, Moritz Weber, for his invaluable guidance and support throughout the research process.
Additionally, this work is a contribution to the SFB-TRR 195.
The present article \lq Complete Classification of Directed Quantum Graphs on $M_2$\rq\ is part of the PhD theses of Nina Kiefer and Björn Schäfer.

We would also like to thank Daniel Gromada, Junichiro Matsuda and Mateusz Wasilewski for comments on an earlier version of this article.

\section{Preliminaries}
\label{sec::preliminaries}

\subsection{Notational Conventions}

Throughout this paper we denote the matrix algebra of $n \times n$ complex matrices by $M_n$ and we write $e_{ij}$ with $i,j = 1, \dots, n$ for its standard generators. The unit matrix of dimension $n\in\N$ is denoted by $\mathrm{I}_n$. Further, we write $\mathrm{Tr}$ for the unnormalized trace on $M_n$ where the dimension $n$ is clear from the context. We denote by $\mathrm{O}(n)$ ($\mathrm{SO}(n)$) the group of (special) orthogonal matrices on $M_n$. The $C^\ast$-algebra of bounded operators on a Hilbert space $\mathcal{H}$ is denoted $B(\mathcal{H})$.

\subsection{Quantum Sets}
\label{sec::pre:quantum_sets}

Let us consider a pair $(B, \psi)$ of a finite-dimensional $C^\ast$-algebra $B$ and a faithful positive functional $\psi$ on $B$. This comes equipped with the \emph{multiplication} map $m: B \otimes B \to B$, $(x,y) \mapsto xy$ and the GNS-inner product $\langle x, y \rangle = \psi(x^\ast y)$ for $x,y \in B$. The latter gives $B$ the structure of a Hilbert space which we denote $L^2(B)$ and which is called the GNS-Hilbert space. With respect to this additional structure, one considers the adjoint $m^\ast: B \to B \otimes B$ of $m$ which is determined by $\langle m^\ast(a), b \otimes c\rangle_{L^2(B) \otimes L^2(B)} = \langle a, bc \rangle_{L^2(B)}$ for all $a, b, c \in B$. This map is called the \emph{comultiplication} on $(B, \psi)$. We prefer to view both $m$ and $m^\ast$ as linear maps between $C^\ast$-algebras and not as bounded operators between Hilbert spaces. Further, the map $\varepsilon: \C \to B$, $\lambda \mapsto \lambda 1_B$ is called the \emph{unit} of $(B, \psi)$ and $\psi: B \to \C$ is called the \emph{counit}.

We call the pair $(B, \psi)$ a \emph{quantum set} if $m m^\ast = \mathrm{Id}_B$. This terminology goes back to \cite[Terminology 3.1]{musto_compositional_2018}, where it is additionally required that the functional $\psi$ is tracial. We do without this latter condition, and call the quantum set $(B, \psi)$ tracial if $\psi$ is tracial. Alternatively to $m m^\ast = \mathrm{Id}_B$ one can ask that $\psi$ is a $\delta$-form, i.e. $\psi(1) = 1$ and $m m^\ast = \delta^2 \mathrm{Id}_B$ for $\delta \geq 0$. This approach is also common in the literature on quantum graphs; see e.g. \cite{daws_quantum_2024,matsuda_classification_2022,gromada_examples_2022,brannan_bigalois_2020,brannan_quantum_2022}. However, we follow \cite{daws_quantum_2024,wasilewski_quantum_2024} in asking $m m^\ast = \mathrm{Id}_B$. If $B = M_n$, then the difference is only a scalar factor as the following proposition shows.

\begin{proposition}[{\cite[Proposition 2.3]{wasilewski_quantum_2024}}]
    \label{pre::prop:delta_form_versus_mm*=Id}
    Let $\psi = \alpha \mathrm{Tr}(\rho \, \cdot)$ with $\rho \in M_n$ and $\alpha \in \C$ be a faithful positive functional on $M_n$. Then
    \begin{align*}
        m^\ast(e_{ij}\rho^{-1}) = \frac{1}{\alpha} \sum_k e_{ik}\rho^{-1} \otimes e_{kj}\rho^{-1}
        \quad \text{ and } \quad
        m m^\ast = \frac{\mathrm{Tr}(\rho^{-1})}{\alpha} \mathrm{Id}_n.
    \end{align*}
    In particular, $\psi$ is a $\delta$-form if and only if we have $m m^\ast = \mathrm{Id}_n$ with respect to the faithful positive functional
    \begin{align*}
        \psi^\prime := \delta^2 \psi.
    \end{align*}
\end{proposition}

\begin{proof}
    The first statement is part of \cite[Proposition 2.3]{wasilewski_quantum_2024} and the second statement follows immediately.
\end{proof}

In this work, we are only interested in quantum sets on $M_2$, i.e. quantum sets of the form $(M_2, \psi)$ for some faithful positive functional $\psi$ with $m m^\ast = \mathrm{Id}_2$. These functionals are classified as follows.

\begin{proposition}
    \label{pre::prop:qsets_on_M2}
    Any faithful positive functional $\psi: M_2 \to \C$ with $mm^\ast = \mathrm{Id}$ is unitarily equivalent to exactly one of the functionals $\psi_q$ with $q \in (0,1]$ given by
    \begin{align*}
        \psi_q: M_2 \to \C, \; X = \sum_{i,j=1}^2 x_{ij} e_{ij} \mapsto  \mathrm{Tr}(\rho_q X) = (1+q^2)(q^{-2} x_{11} + x_{22}),
    \end{align*}
    where 
    $$
        \rho_q = (1+q^2) \mathrm{diag}(q^{-2}, 1) = \begin{pmatrix}
            \frac{1+q^2}{q^2} & 0 \\
            0 & 1+q^2
        \end{pmatrix}.
    $$
    We write $\tau := \psi_1$ and note $\tau = 2 \mathrm{Tr}$. 
\end{proposition}

\begin{proof}
    It is well known (see e.g. \cite{matsuda_classification_2022,soltan_quantum_2010}) that every faithful state on $M_2$ is unitarily equivalent to one of the Powers states
    \begin{align*}
        &\omega_q: M_2 \to \C, X = \sum_{i,j=1}^2 x_{ij} e_{ij} \mapsto \frac{1}{1+q^2}(x_{11} + q^2 x_{22}). 
    \end{align*}
    where $q \in (0,1]$. 
    These states are $\delta$-forms with $\delta = q + q^{-1}$, see e.g. \cite[Lemma 3.4]{matsuda_classification_2022}.
    Thus, every faithful positive functional $\psi$ on $M_2$ with $m m^\ast = \mathrm{Id}_2$ is up to unitary equivalence of the form $(q+q^{-1})^2 \omega_q$ for some $q \in (0,1]$. The statement follows by observing
    \begin{align*}
         \frac{(q+q^{-1})^2}{1+q^2}(x_{11} + q^2 x_{22})
         = \frac{(q^{-1}(1+q^2))^2}{1+q^2}(x_{11} + q^2 x_{22})
         = (1+q^2)(q^{-2} x_{11} + x_{22}).
    \end{align*}
\end{proof}

For $q = 1$ one obtains the tracial quantum space $(M_2, \tau)$, while for $q < 1$ the functional $\psi_q$ is not tracial. This has a significant impact on the automorphism group of $(M_2, \psi_q)$ which we collect in the next proposition. The automorphism group of a quantum set $(B, \psi)$ consists of all $\ast$-automorphisms $\pi$ of the $C^\ast$-algebra $B$ which preserve the functional $\psi$ in the sense that $\psi \circ \pi = \psi$. In Definition \ref{pre::def:isomorphism_of_qgraphs} we will more generally introduce isomorphisms between quantum graphs.

Recall the Pauli matrices
\begin{align*}
    \sigma_0=\mathrm{I}_2, \quad  
    \sigma_1=\begin{pmatrix}0&1\\1&0\end{pmatrix}, \quad 
    \sigma_2=\begin{pmatrix}0&-i\\i&0\end{pmatrix} \quad \text{ and } \quad 
    \sigma_3=\begin{pmatrix}1&0\\0&-1\end{pmatrix}.
\end{align*}

\begin{proposition}
\label{pre::prop:aut_group_of_(M2,psi)}    
    The automorphism group of $(M_2, \psi_q)$ is up to isomorphism
    \begin{itemize}
        \item $\SO(3)$, the special orthogonal group in three dimensions if $q=1$ with group action given by
        \begin{align*}
            R \cdot (x_0 \sigma_0 + x_1 \sigma_1 + x_2 \sigma_2 + x_3 \sigma_3) = x_0 \sigma_0 + x_1^\prime \sigma_1 + x_2^\prime \sigma_2 + x_3^\prime \sigma_3
        \end{align*}
        where $\sigma_0, \sigma_1, \sigma_2, \sigma_3$ are the Pauli matrices, $R \in \SO(3)$, $x_0, x_1, x_2, x_3 \in \C$, $x_1^\prime, x_2^\prime, x_3^\prime \in \C$ and 
        $$
        \begin{pmatrix}
            x_1^\prime \\ x_2^\prime \\ x_3^\prime
        \end{pmatrix}
        = R \begin{pmatrix}
            x_1 \\ x_2 \\ x_3
        \end{pmatrix},
        $$
        \item $\SO(2)$, the special orthogonal group in two dimensions if $q\in(0,1)$ with group action given by
        \begin{align*}
            R \cdot (x_0 \sigma_0 + x_1 \sigma_1 + x_2 \sigma_2 + x_3 \sigma_3) = x_0 \sigma_0 + x_1^\prime \sigma_1 + x_2^\prime \sigma_2 + x_3 \sigma_3
        \end{align*}
        where $R \in \SO(2)$, $x_0, x_1, x_2, x_3, x_1^\prime, x_2^\prime \in \C$ and 
        \begin{align*}
            \begin{pmatrix}
                x_1^\prime \\ x_2^\prime
            \end{pmatrix}
            = R
            \begin{pmatrix}
                x_1 \\ x_2
            \end{pmatrix}.
        \end{align*}
    \end{itemize}
\end{proposition}

\begin{proof}
    For the first statement, note that every isomorphism preserves the trace. Therefore, the automorphism group of $(M_2, \tau)$ is nothing else but the automorphism group of $M_2$. It is well known that this group is $\SO(3)$ with the action from above, see e.g. \cite{soltan_quantum_2010}.
    
    For the second claim let $q \in (0,1)$.
    Since $\psi_q(e_{11}) \neq \psi_q(e_{22})$, any automorphism $\varphi: M_2 \to M_2$ of $(M_2, \psi_q)$ must preserve the two projections $e_{11}$ and $e_{22}$. Writing $\varphi$ as conjugation $U \cdot U^\ast$ for a unitary matrix $U \in \mathrm{U}(2)$, it follows that $U$ is diagonal. By multiplying with a suitable scalar if necessary, we may assume 
    \begin{align*}
        U = \mathrm{diag}(\lambda, 1)
    \end{align*}
    for some $\lambda = a + ib \in S^1$. Then one checks 
    \begin{align*}
        U \sigma_i U^\ast = \sigma_i
    \end{align*}
    for $i=0,3$, as well as
    \begin{align*}
        U \sigma_1 U^\ast = \begin{pmatrix}
            0 & \overline{\lambda} \\
            \lambda & 0
        \end{pmatrix}
        = a \sigma_1 + b \sigma_2
        \quad \text{ and } \quad
        U \sigma_2 U^\ast = \begin{pmatrix}
            0 & -i \overline{\lambda} \\
            i \lambda & 0
        \end{pmatrix}
        = - b \sigma_1 + a \sigma_2.
    \end{align*}
\end{proof}

\subsection{Quantum Graphs}

The main objects of this paper are quantum graphs. They can be defined in three different ways, and we  present these different definitions without giving preference to one. To begin with, we discuss general quantum graphs on quantum sets $(B, \psi)$, but later we focus on quantum graphs on a matrix algebra $M_n$.

\begin{definition}
    \label{pre::def:qgraph}\label{def:quantumGraph}  
    A \emph{quantum graph} $\mathcal{G}$ on a quantum set $(B, \psi)$ is given by either of the following:
    \begin{enumerate}
        \item given an embedding $B \subset B(\mathcal{H})$, a $B^\prime$-$B^\prime$-bimodule\footnote{We denote with $B^\prime$ the commutant of $B$ in $B(\mathcal{H})$. Then $S \subset B(\mathcal{H})$ is a $B^\prime$-$B^\prime$-bimodule if and only if, it is an operator space that satisfies $B^\prime S = \{b^\prime x: b^\prime \in B^\prime, x \in S\} \subset S$, and similarly $S B^\prime \subset S$.}
         $S \subset B(\mathcal{H})$ which we call a \emph{quantum edge space},
        \item a linear map $A \cc B \to B$ such that $A(x^\ast) = A(x)^\ast$ holds for all $x \in B$ and 
        \begin{align*}
            m(A\otimes A)m^*=A,
        \end{align*}
        which we call a \emph{quantum adjacency matrix}\footnote{In the presence of $m(A \otimes A)m^\ast = A$ the first condition $A(x^\ast)=A(x)^\ast$ is equivalent to $A$ being completely positive, see \cite{matsuda_classification_2022}.},
        \item a projection $P \in B \otimes B^{\mathrm{op}}$ 
        which we call a \emph{quantum edge projection}\footnote{$B^{\mathrm{op}}$ is the opposite algebra of $B$, 
        meaning that it is a copy of the set $B$ together with the natural vector space structure and multiplication given by $a^{\mathrm{op}} b^{\mathrm{op}} = (ba)^{\mathrm{op}}$.
        Following \cite{wasilewski_quantum_2024} 
        we denote the elements of $B^{\mathrm{op}}$ by 
        $b^{\mathrm{op}}$
        for $b \in B$.}.
    \end{enumerate}
    If the quantum set $(B, \psi)$ is tracial we call the quantum graph $\mathcal{G}$ tracial as well. Further, we call the dimension of $S$ as a vector space the number of \emph{quantum edges} of $\mathcal{G}$, and refer to \cite[Section 2.3]{gromada_examples_2022} for a detailed discussion of quantum edges of quantum graphs.
\end{definition}

There are 1-1-correspondences between quantum edge spaces, adjacency matrices and edge projections. These go back to \cite{musto_compositional_2018, weaver_quantum_2012} and have been discussed in various places. In the case of tracial or nontracial quantum sets different correspondences have been described by Daws and Wasilewski \cite{daws_quantum_2024,wasilewski_quantum_2024}. In this paper, we consistently use the framework of Wasilewski.
Then, the applicable correspondences are compactly explained in \cite[Section 3.4]{wasilewski_quantum_2024}.

In fact, we will only be interested in quantum sets of the form $(M_n, \psi)$, and we further specialize on the case where the quantum edge space $S \subset B(\mathcal{H})$ is considered with respect to the natural embedding $M_n \subset B(\C^n)$. In particular, the condition that $S$ is a $B^\prime$-$B^\prime$-submodule simplifies to asking that $S$ is some operator space in $M_n$ since $B^\prime = \C$. 

Before proceeding, we need to discuss alternative inner products on $M_n$ besides the GNS-inner product induced by the faithful positive functional $\psi$. Evidently, there is a positive definite matrix $\rho$ such that $\psi = \mathrm{Tr}(\rho\, \cdot)$. Then, the KMS-inner product induced by $\psi$ is given by
\begin{align*}
    \langle x, y \rangle^{KMS}_\psi = \psi\left(x^\ast \rho^{\frac{1}{2}} y \rho^{-\frac{1}{2}}\right) = \mathrm{Tr}(x^\ast \rho^{\frac{1}{2}} y \rho^{\frac{1}{2}}).
\end{align*}
Further, we will need the KMS-inner product induced by the functional 
$$
    \psi^{-1} = \frac{\mathrm{Tr}(\rho^{-1} \, \cdot)}{\mathrm{Tr}(\rho^{-1})} .
$$
Since $m m^\ast = \mathrm{Id}_n$ entails $\mathrm{Tr}(\rho^{-1}) = 1$ (see Proposition \ref{pre::prop:delta_form_versus_mm*=Id}), this is given by
\begin{align*}
    \langle x, y \rangle^{KMS}_{\psi^{-1}} = \mathrm{Tr}(x^\ast \rho^{-\frac{1}{2}} y \rho^{-\frac{1}{2}}).
\end{align*}
Now, the correspondence between quantum edge spaces $S$ and quantum adjacency matrices $A$ is as follows.

\begin{proposition}[{$S\mapsto A$, \cite[Proposition 3.30]{wasilewski_quantum_2024}}]
    \label{pre::prop:S->A}
    Let $(X_i)_i$ be an orthonormal basis of the quantum edge space $S\subset M_n$ with respect to the KMS-inner product induced by $\psi^{-1}$. Then, the map
    \begin{align*} 
        A\cc M_n \to M_n, \quad x \mapsto
        \sum_i\rho^{-\frac{1}{4}}X_i\rho^{\frac{1}{4}}x\rho^{\frac{1}{4}}X_i^*\rho^{-\frac{1}{4}}
    \end{align*}
    is a quantum adjacency matrix, and the assignment $S \mapsto A$ gives a well-defined\footnote{This means that $A$ does only depend on the space $S$ and not on the choice of KMS-orthonormal basis.} 1-1-correspondence between quantum edge spaces and quantum adjacency matrices.
\end{proposition}

\begin{remark}
    \label{pre::rem:adjacency_matrix_A_as_Kronecker_product}
    The linear spaces $M_n$ and $\C^{n^2}$ are naturally identified via the map $\mathrm{vec}: M_n \ni e_{ij} \mapsto f_{(i-1)n + j} \in \C^{n^2}$ where the $f_i$ with $i=1,\dots,n^2$ are the standard generators of $\C^{n^2}$. Then, a linear map on $M_n$ is identified with a linear map on $\C^{n^2}$ which in turn is nothing else but an $n^2 \times n^2$-matrix. Generally, one has for any $R, S, X \in M_n$
    \begin{align*}
        \mathrm{vec}(RXS) = (R \otimes S^t) \mathrm{vec}(X),
    \end{align*}
    see e.g. \cite[Section 1.1.2]{watrous_theory_2018}.
    Here $S^t$ is the transposed of $S$ and $\otimes$ denotes the Kronecker product given by
    \begin{align*}
        [R \otimes S]_{(i-1)n+j, (k-1)n+\ell} = R_{ik} S_{j \ell}.
    \end{align*}
    Thus we can identify the adjacency matrix $A: M_n \to M_n$ from the previous proposition with the $n^2 \times n^2$-matrix
    \begin{align*}
        A &\vcentcolon= \sum_i (\rho^{-\frac{1}{4}}X_i\rho^{\frac{1}{4}}) \otimes (\rho^{\frac{1}{4}}X_i^*\rho^{-\frac{1}{4}})^t 
            = \sum_i (\rho^{-\frac{1}{4}}X_i\rho^{\frac{1}{4}}) \otimes \overline{(\rho^{-\frac{1}{4}}X_i \rho^{\frac{1}{4}})}.
    \end{align*}
\end{remark}

The next proposition describes the correspondence between quantum edge spaces and quantum edge projections.

\begin{proposition}[{$S\mapsto P$, \cite[Proposition 3.29]{wasilewski_quantum_2024}}]
    \label{pre::prop:S->P}
    Let $(X_i)_i$ be an orthonormal basis of the quantum edge space $S\subset M_n$ with respect to the KMS inner product induced by $\psi^{-1}$ and define $P\in M_n\otimes M_n^{\op}$ as\footnote{Note that the formula in \cite[Proposition 3.29]{wasilewski_quantum_2024} contains a scalar factor $\frac{1}{\mathrm{Tr}(\rho^{-1})}$. However, as $m m^\ast = \mathrm{Id}_n$, we have $\mathrm{Tr}(\rho^{-1}) = 1$, and therefore we can do without this factor.}
    \begin{align*}
        P = \sum_{i,k,l}\rho^{-\frac{1}{4}}X_ie_{kl}\rho^{-\frac{1}{4}}\otimes
        \left(\rho^{-\frac{1}{4}}X_i^*e_{lk}\rho^{-\frac{1}{4}}\right)^{\op}.
    \end{align*}
    
    Then $P$ is a quantum edge projection, and the assignment $S \mapsto P$ gives a well-defined 1:1-correspondence between quantum edge spaces and quantum edge projections.
\end{proposition}

\begin{example}
\label{pre::example:empty_trivial_complete_qgraph}
Let us present three easy examples of quantum graphs.
\begin{enumerate}
    \item 
    The \emph{empty quantum graph} on $(M_n, \psi)$ is given by the quantum edge space $S=\{0\}$. 
    \item
    The \emph{trivial quantum graph} on $(M_n, \psi)$ is given by the quantum edge space $S = \spann \{\mathrm{I}_n\}$ or equivalently by the quantum adjacency matrix $A = \mathrm{I}_{n^2}$. It is a reflexive and undirected quantum graph with one quantum edge.
    \item 
    The \emph{complete quantum graph} on $(M_n, \psi)$ is given by the quantum edge space $S = M_n$. This is quantum graph with the maximal number of $n^2$ quantum edges.
\end{enumerate}
\end{example}

\subsection{Properties of Quantum Graphs}
\label{pre::subsec:properties_of_quantum_graphs}

There are two important pairs of properties of quantum graphs which are sometimes part of their definition. These are being directed/undirected and being reflexive/loopfree. For instance \cite{daws_quantum_2024,musto_compositional_2018} asks that their quantum graphs be always undirected while \cite{matsuda_classification_2022} focuses in his classification on undirected and reflexive quantum graphs. At this point, we only discuss these properties for quantum graphs on quantum sets of the form $(M_n, \psi)$ for we are only interested in quantum graphs on $M_2$. As before, let $\rho \in M_n$ be the positive matrix with $\psi(\cdot) = \mathrm{Tr}(\rho \, \cdot)$.

Let us first present the definition of reflexive and loopfree quantum graphs. This definition is well-known, see e.g. \cite[Definion 2.19]{matsuda_classification_2022}, \cite[Section 1.3]{gromada_examples_2022}.

\begin{definition}
\label{pre::def:(ir)reflxive_qgraphs}
    Let $\mathcal{G}$ be a quantum graph on $(M_n, \psi)$ with quantum edge space $S \subset M_n$, quantum adjacency matrix $A: M_n \to M_n$ and quantum edge projection $P \in M_n \otimes M_n^\mathrm{op}$. We say that $\mathcal{G}$ is \emph{reflexive} (or \emph{has all loops}) if one of the following equivalent conditions is satisfied.
    \begin{enumerate}
        \item $\mathrm{I}_n\in S$,
        \item $m(A\otimes\Id_n)m^\ast=\Id_n$, 
        \item $m((\mathrm{Id} \otimes \sigma_{\frac{i}{2}})P) = \mathrm{I}_n$ where $\sigma_z: M_n \to M_n, x \mapsto \rho^{iz} x \rho^{-iz}$ for all $z \in \C$.
    \end{enumerate}
    On the other hand, we say that $\mathcal{G}$ is \emph{loopfree} (or \emph{irreflexive}) if one of the following holds.
    \begin{enumerate}
        \item $\mathrm{I}_n\in S^\bot$ where the complement is taken with respect to the KMS-inner product induced by $\psi^{-1}$,\footnote{Using the definition of $\psi^{-1}$ this is equivalent to $\Tr(x \rho^{-1}) = 0$ for all $x \in S$.}
        \item $m(A\otimes\Id_n)m^\ast=0$, 
        \item $m((\mathrm{Id} \otimes \sigma_{\frac{i}{2}})P) = 0$.
    \end{enumerate}
\end{definition}

\begin{proposition}
    The previous definition is well-defined.
\end{proposition}

\begin{proof}
    We show that the three conditions are equivalent, respectively. 
    
    For ``(1) $\Leftrightarrow$ (2)'' let $(X_i)_{i=0}^s$ be an orthonormal basis of $S$ with respect to the KMS-inner product induced by $\psi^{-1}$. Without loss of generality, we may assume $\langle X_i, \mathrm{I}_n \rangle_{\psi^{-1}}^{KMS} = 0$ for all $i > 0$. Then, one has
    $\mathrm{I}_n \in S$ if and only if $X_0 = \mathrm{I}_n$, and $\mathrm{I}_n \in S^\bot$ if and only if $\langle X_0, \mathrm{I}_n \rangle_{\psi^{-1}}^{KMS} = 0$.
    \comments{
    \begin{align*}
        & \mathrm{I}_n \in S & \Leftrightarrow & \qquad \qquad X_0 = \mathrm{I}_n, \\ 
        & \mathrm{I}_n \in S^\bot & \Leftrightarrow & \qquad \qquad \langle X_0, \mathrm{I}_n \rangle_{\psi^{-1}}^{KMS} = 0. \\
    \end{align*}
    }
    Using Propositions \ref{pre::prop:delta_form_versus_mm*=Id} and \ref{pre::prop:S->A} one has for all matrix units $e_{ac} \in M_n$
    \begin{align*}
        m(A \otimes \mathrm{Id}_n)m^\ast(e_{ac}\rho^{-1})
            &= \sum_{b}  A(e_{ab}\rho^{-1}) e_{bc} \rho^{-1} \\
            &= \sum_b \left( \sum_i \rho^{-\frac{1}{4}} X_i \rho^{\frac{1}{4}} \left( e_{ab} \rho^{-1} \right) \rho^{\frac{1}{4}} X_i^\ast  \rho^{-\frac{1}{4}} \right) e_{bc} \rho^{-1} \\
            &= \sum_i \rho^{-\frac{1}{4}} X_i \rho^{\frac{1}{4}} \left( \sum_b e_{ab} \rho^{-\frac{3}{4}} X_i^\ast  \rho^{-\frac{1}{4}} e_{bc} \right) \rho^{-1}.
    \end{align*}
    For every index $i$ one observes
    \begin{align*}
        \sum_b e_{ab} \rho^{-\frac{3}{4}} X_i^\ast  \rho^{-\frac{1}{4}} e_{bc}
            &=\mathrm{Tr}\left( \rho^{-\frac{3}{4}} X_i^\ast  \rho^{-\frac{1}{4}} \right) e_{ac}\\
            &=\mathrm{Tr}\left(X_i^\ast  \rho^{-\frac{1}{2}} \mathrm{I}_n \rho^{-\frac{1}{2}} \right) e_{ac}\\
            &= \langle X_i, \mathrm{I}_n \rangle_{\psi^{-1}}^{KMS} e_{ac}.
    \end{align*}
    Using $\langle X_i, \mathrm{I}_n \rangle_{\psi^{-1}}^{KMS} = 0$ for $i > 0$ it follows
    \begin{align*}
        m(A \otimes \mathrm{Id}_n)m^\ast(e_{ac}\rho^{-1})
        =\rho^{-\frac{1}{4}} X_0 \rho^{\frac{1}{4}} \langle X_0, \mathrm{I}_n \rangle_{\psi^{-1}}^{KMS} e_{ac} \rho^{-1}.
    \end{align*}
    \comments{
    \begin{align*}
        m(A \otimes \mathrm{Id}_n)m^\ast(e_{ac}\rho^{-1})
            &= \sum_{b}  A(e_{ab}\rho^{-1}) e_{bc} \rho^{-1} \\
            &= \sum_b \left( \sum_i \rho^{-\frac{1}{4}} X_i \rho^{\frac{1}{4}} \left( e_{ab} \rho^{-1} \right) \rho^{\frac{1}{4}} X_i^\ast  \rho^{-\frac{1}{4}} \right) e_{bc} \rho^{-1} \\
            &= \sum_i \rho^{-\frac{1}{4}} X_i \rho^{\frac{1}{4}} \left( \sum_b e_{ab} \rho^{-\frac{3}{4}} X_i^\ast  \rho^{-\frac{1}{4}} e_{bc} \right) \rho^{-1} \\
            &= \sum_i \rho^{-\frac{1}{4}} X_i \rho^{\frac{1}{4}} \mathrm{Tr}\left( \rho^{-\frac{3}{4}} X_i^\ast  \rho^{-\frac{1}{4}} \right) e_{ac} \rho^{-1} \\
            &= \sum_i \rho^{-\frac{1}{4}} X_i \rho^{\frac{1}{4}} \mathrm{Tr}\left(X_i^\ast  \rho^{-1} \right) e_{ac} \rho^{-1} \\
            &= \sum_i \rho^{-\frac{1}{4}} X_i \rho^{\frac{1}{4}} \langle X_i, \mathrm{I}_n \rangle_{\psi^{-1}}^{KMS} e_{ac} \rho^{-1} \\
            &= \rho^{-\frac{1}{4}} X_0 \rho^{\frac{1}{4}} \langle X_0, \mathrm{I}_n \rangle_{\psi^{-1}}^{KMS} e_{ac} \rho^{-1}.
    \end{align*}
    }
    Consequently, we have $m(A \otimes \mathrm{Id}_n) m^\ast = \mathrm{Id}_n$ if and only if, $\mathrm{I}_n \in S$, and similarly $m(A \otimes \mathrm{Id}_n) m^\ast = 0$ if and only if, $\mathrm{I}_n \in S^\bot$.
    
    The equivalence ``(1) $\Leftrightarrow$ (3)'' can be seen very similarly. Indeed, with Proposition \ref{pre::prop:S->P} and $(X_i)_{i=0}^s$ as above, it suffices to note
    \comments{
    \begin{align*}
        m((\mathrm{Id} \otimes \sigma_{\frac{i}{2}})P)
        &= m\left(\sum_{i,k,l}\rho^{-\frac{1}{4}}X_ie_{kl}\rho^{-\frac{1}{4}}\otimes
        \left(\rho^{-\frac{1}{2}} \left( \rho^{-\frac{1}{4}}X_i^*e_{lk}\rho^{-\frac{1}{4}} \right) \rho^{\frac{1}{2}}\right)^{\op}\right) \\
        &= \sum_{i,k,l} \rho^{-\frac{1}{4}}X_ie_{kl}\rho^{-1}X_i^*e_{lk}\rho^{\frac{1}{4}} 
    \end{align*}
    By using that $\sum_le_{kl}\rho^{-1}X_i^*e_{lk}=\mathrm{Tr}\left(X_i^* \rho^{ -\frac{1}{2}} \mathrm{I}_n \rho^{-\frac{1}{2}}\right) e_{kk}$ we get
    \begin{align*}
        m(P_\psi)
        =\rho^{-\frac{1}{4}}X_0\rho^{\frac{1}{4}} \langle X_0, \mathrm{I}_n \rangle_{\psi^{-1}}^{KMS}.
    \end{align*}
    }
    \begin{align*}
        m((\mathrm{Id} \otimes \sigma_{\frac{i}{2}})P)
        &= m\left(\sum_{i,k,l}\rho^{-\frac{1}{4}}X_ie_{kl}\rho^{-\frac{1}{4}}\otimes
        \left(\rho^{-\frac{1}{2}} \left( \rho^{-\frac{1}{4}}X_i^*e_{lk}\rho^{-\frac{1}{4}} \right) \rho^{\frac{1}{2}}\right)^{\op}\right) \\
        &= \sum_{i,k,l} \rho^{-\frac{1}{4}}X_ie_{kl}\rho^{-1}X_i^*e_{lk}\rho^{\frac{1}{4}} \\
        &= \rho^{-\frac{1}{4}}X_0\rho^{\frac{1}{4}} \langle X_0, \mathrm{I}_n \rangle_{\psi^{-1}}^{KMS}.
    \end{align*}
\end{proof}

The definition of undirected quantum graphs is well-known in the tracial case. For quantum graphs on nontracial quantum sets, the definition can be based on either the GNS-inner product or alternatively the KMS-inner product. The former approach has been discussed in \cite{daws_quantum_2024}, while the use of the KMS-inner product has been pioneered in \cite{wasilewski_quantum_2024}. 

\begin{definition}
\label{pre::def:GNS-KMS-undirected_quantum_graphs}
    Let $\mathcal{G}$ be a quantum graph on $(M_n, \psi)$. We say that $\mathcal{G}$ is \emph{GNS-undirected} if one of the following equivalent conditions is satisfied:
    \begin{enumerate}
        \item $S=S^*$ and $\rho S \rho^{-1} = S$
        \item $\sigma(P)=P$ and $P = (\sigma_z \otimes \sigma_z)(P)$ for all $z \in \C$, where 
        $$
        \sigma(x \otimes y) = x \otimes y
        \text{ and }
        \sigma_z(x) = \rho^{iz} x \rho^{-iz}
        $$
        for all $x, y \in M_n$ and $z \in \C$,
        \item $A=A^*$, where $A^*$ is the adjoint of $A$ viewed as a bounded operator on the GNS-space $L^2(M_n, \psi)$.
    \end{enumerate}
    On the other hand, $\mathcal{G}$ is called \emph{KMS-undirected} if one of the following equivalent conditions is true:
    \begin{enumerate}
        \item $S = S^*$,
        \item $\sigma(P) = P$,
        \item $A = A^*$, where $A^*$ is the adjoint of $A$ viewed as a bounded operator on the KMS-space $L^2_{KMS}(M_n, \psi)$.
    \end{enumerate}
    If $\psi$ is tracial, $\mathcal{G}$ is GNS-undirected if and only if it is KMS-undirected, and in this case we simply call $\mathcal{G}$ undirected.
\end{definition}

\begin{proposition}
    The previous definition is well-defined.
\end{proposition}

\begin{proof}
    The equivalence of the three statements for KMS-undirectedness can be directly found in Wasilewski's paper \cite{wasilewski_quantum_2024}. Indeed, the equivalence between (2) and (3) follows from \cite[Proposition 3.7 (iv)]{wasilewski_quantum_2024}. Further, (1) $\Leftrightarrow$ (2) follows from \cite[Proposition 3.29]{wasilewski_quantum_2024}.
    

    It remains to show that the three statements for GNS-undirectedness are equivalent. This has been proved by Daws in \cite{daws_quantum_2024}, but we need to reprove it here in Wasilewski's framework. The following proof was suggested to us by Wasilewski.
    
    We use the observation that a GNS-self-adjoint linear map commutes with all modular operators $\sigma_z$ for $z \in \C$ \cite[Lemma 2.1]{wasilewski_quantum_2024}. It follows easily that $A$ is GNS-self-adjoint if and only if it is KMS-self-adjoint and commutes with all $\sigma_z$. In view of the equivalent characterizations of GNS-undirectedness, it only remains to prove that the following are equivalent:
    \begin{enumerate}
        \item[(a)] $A$ commutes with all $\sigma_z$,
        \item[(b)] $S = \rho S \rho^{-1}$,
        \item[(c)] $P = (\sigma_z \otimes \sigma_z)(P)$ for all $z \in \C$.
    \end{enumerate}

    Using functional calculus one checks that these statements are equivalent to a) $A$ commutes with $\sigma_1$, b) $S = \rho^i S \rho^{-i}$ and c) $P = (\sigma_1 \otimes \sigma_1)(P)$, respectively. We show that these are equivalent.

    \underline{(a) $\Leftrightarrow$ (b)} Let $\{X_j\}$ be an ONB of $S$. Then we get from Proposition \ref{pre::prop:S->A} and (a)
    \begin{align*}
        A(x) = (\sigma_{1} \circ A \circ \sigma_{-1})(x)
            = \sum_j\rho^i \rho^{-\frac{1}{4}} X_j \rho^{\frac{1}{4}}\rho^{-i} x \rho^i \rho^{\frac{1}{4}}X_j^*\rho^{-\frac{1}{4}} \rho^{-i}.
    \end{align*}
    It is not hard to check that $\{\rho^i X_j \rho^{-i}\}$ is again an ONB with respect to the KMS-inner product induced by $\psi_{-1}$. Hence, one obtains $S = \rho^i S \rho^{-i}$ from Proposition \ref{pre::prop:S->A}. 
    Conversely, the same formula as above shows that $S = \rho^i S \rho^{-i}$ entails that $A$ commutes with $\sigma_{i}$. 

    \underline{(b) $\Leftrightarrow$ (c)} Let $\{X_j\}$ be an ONB of $S$ as above. From Proposition \ref{pre::prop:S->P} and (b) we get
    \begin{align*}
        (\sigma_{1} \otimes \sigma_{1})(P) = \sum_{j,k,l}\rho^i \rho^{-\frac{1}{4}}X_je_{kl}\rho^{-\frac{1}{4}}\rho^{-i} \otimes
        \left(\rho^i \rho^{-\frac{1}{4}}X_j^*e_{lk}\rho^{-\frac{1}{4}}\rho^{-i}\right)^{\op} 
        = P.
    \end{align*}
    Thus, $P$ commutes with $(\rho^i \otimes \rho^i)$. 
    Conversely, one can use $(\sigma_{1} \otimes \sigma_{1})(P) = P$ to prove $S = \rho^i S \rho^{-i}$ similarly as above.
\end{proof}

\begin{example}
\label{pre::example:properties}
Recall the examples of Example \ref{pre::example:empty_trivial_complete_qgraph}. Both the trivial and complete quantum graph on $M_n$ are reflexive, GNS-undirected and KMS-undirected.
\end{example}

\section{Quantum Graph Isomorphism}
\label{sec::qgraph_isomorphism}


In this section, we discuss when two quantum graphs are isomorphic. This is not be confused with quantum isomorphism of quantum graphs.
Note that we described the automorphism groups of $(M_2, \tau)$ and $(M_2, \psi_q)$ for $q \in (0,1)$ already in Proposition \ref{pre::prop:aut_group_of_(M2,psi)}. 

\begin{definition}
    \label{pre::def:isomorphism_of_qgraphs}
    Let $\mathcal{G}$ and $\hat{\mathcal{G}}$ be quantum graphs on quantum sets $(B, \psi)$ and $(\hat{B}, \hat{\psi})$ with quantum edge spaces $S$, $\hat{S}$, quantum adjacency matrices $A$, $\hat{A}$ and quantum edge projections $P$, $\hat{P}$, respectively. A $\ast$-isomorphism $\pi: B \to \hat{B}$ is 
    \begin{enumerate}
        \item an \emph{isomorphism of the quantum sets $(B, \psi)$ and $(\hat{B}, \hat{\psi})$} if
        \begin{align*}
            \hat{\psi} \circ \pi = \psi,
        \end{align*}
        \item an \emph{isomorphism of the quantum graphs $\mathcal{G}$ and $\hat{\mathcal{G}}$} if additionally one of the following equivalent conditions holds:
        \begin{enumerate}
            \item $\pi(S) = \hat{S}$,
            \item $\pi \circ A = \hat{A} \circ \pi$,
            \item $(\pi \otimes \pi^\mathrm{op}) P = \hat{P}$.
        \end{enumerate}
    \end{enumerate}
\end{definition}

\begin{proposition}
    The previous definition is well-defined.
\end{proposition}

We only prove the equivalence of the three conditions for quantum graph isomorphisms in the case where $B = 
\hat{B} = M_n$ and the quantum edge spaces $S$, $\hat{S}$ are chosen with respect to the embedding $M_n \subset B(\C^n)$. Indeed, this is the only case that we will use later on. 
It is no problem to prove the statement more generally.

\begin{proof}
    Let $U \in M_n$ be a unitary matrix with $\pi(x) = U x U^\ast$ for all $x \in M_n$, and let $(X_i)_i$ be an orthonormal basis of $S$ with respect to the KMS-inner product induced by $\psi^{-1}$. Further, let $\rho, \hat{\rho} \in M_n$ be positive definite matrices with
    \begin{align*}
        \psi = \mathrm{Tr}(\rho \, \cdot), \quad \hat{\psi} = \mathrm{Tr}(\hat{\rho} \, \cdot).
    \end{align*}
    Observe $U \rho = \hat{\rho} U$ since
    \begin{align*}
        \mathrm{Tr}(\rho \, \cdot) = \psi = \hat{\psi} \circ \pi = \mathrm{Tr}(\hat{\rho} U \, \cdot U^\ast) = \mathrm{Tr}(U^\ast \hat{\rho} U \, \cdot)
    \end{align*}
    entails $\rho = U^\ast \hat{\rho} U$. Further, observe that two matrices $X, Y \in M_n$ are orthogonal with respect to $\langle \cdot, \cdot \rangle^{KMS}_{\psi^{-1}}$ if and only if $UXU^\ast$ and $UYU^\ast$ are orthogonal with respect to $\langle \cdot, \cdot \rangle^{KMS}_{\hat{\psi}^{-1}}$. Indeed,
    \begin{align*}
        \langle UXU^\ast, UYU^\ast \rangle^{KMS}_{\hat{\psi}^{-1}}
            &= \mathrm{Tr}(UX^\ast U^\ast \hat{\rho}^{-\frac{1}{2}} U Y U^\ast \hat{\rho}^{-\frac{1}{2}}) \\
            &= \mathrm{Tr}(UX^\ast U^\ast U \rho^{-\frac{1}{2}} Y \rho^{-\frac{1}{2}} U^\ast) \\
            &= \langle X, Y \rangle^{KMS}_{\psi^{-1}}.
    \end{align*}
    Thus, $(X_i)_i \subset M_n$ is an orthonormal basis of $S$ with respect to the KMS-inner product induced by $\psi^{-1}$ if and only if $(\pi(X_i))_i = (U X_i U^\ast)_i$ is an orthonormal basis of $\pi(S)$ with respect to the KMS-inner product induced by $\hat{\psi}^{-1}$.
    
    To prove ``(a) $\Rightarrow$ (b)'' assume $\hat{S} = \pi(S)$ and recall from Proposition \ref{pre::prop:S->A}
    \begin{align*}
        A(x) &= \sum_i\rho^{-\frac{1}{4}}X_i\rho^{\frac{1}{4}}x\rho^{\frac{1}{4}}X_i^*\rho^{-\frac{1}{4}}, \\
        \hat{A}(x) &= \sum_i \hat{\rho}^{-\frac{1}{4}}(U X_i U^\ast) \hat{\rho}^{\frac{1}{4}}x\hat{\rho}^{\frac{1}{4}} (U X_i^* U^\ast )\hat{\rho}^{-\frac{1}{4}}.
    \end{align*}
    Using again $U \rho = \hat{\rho} U$ one obtains for all $x \in M_n$
    \begin{align*}
        \pi \circ A(x)
            &= \sum_i U \rho^{-\frac{1}{4}} X_i \rho^{\frac{1}{4}} x \rho^{\frac{1}{4}}X_i^*\rho^{-\frac{1}{4}} U^\ast \\
            &= \sum_i \hat{\rho}^{-\frac{1}{4}} U X_i U^\ast \hat{\rho}^{\frac{1}{4}} U x U^\ast \hat{\rho}^{\frac{1}{4}} U X_i^* U^\ast \hat{\rho}^{-\frac{1}{4}} \\
            &= \hat{A} \circ \pi(x).
    \end{align*}

    For ``(b) $\Rightarrow$ (a)'' observe
    \begin{align*}
        \hat{A}(x) 
            &= \hat{A} \circ \pi \circ \pi^{-1}(x) \\
            &= \pi \circ A \circ \pi^{-1}(x) \\
            &= \sum_i U \rho^{-\frac{1}{4}}X_i\rho^{\frac{1}{4}} (U^\ast x U) \rho^{\frac{1}{4}}X_i^*\rho^{-\frac{1}{4}} U^\ast \\
            &= \sum_i \hat{\rho}^{-\frac{1}{4}} UX_i U^\ast \hat{\rho}^{\frac{1}{4}} x \hat{\rho}^{\frac{1}{4}} U X_i^* U^\ast \hat{\rho}^{-\frac{1}{4}},
    \end{align*}
    where $(UX_iU^\ast)_i$ is an orthonormal set with respect to the KMS-inner product induced by $\hat{\psi}^{-1}$. Consequently, it follows from Proposition \ref{pre::prop:S->A} that 
    $$
    \hat{S} = \mathrm{span}\left\{ UX_iU^\ast: i \right\} = \pi(S).
    $$

    The equivalence ``(a) $\Leftrightarrow$ (c)'' is shown very similarly. Indeed, it suffices to observe with Proposition \ref{pre::prop:S->P}
    \begin{align*}
        (\pi \otimes \pi^\mathrm{op}) P
            &=
            \sum_{i,k,\ell} U \rho^{-\frac{1}{4}}X_ie_{k\ell}\rho^{-\frac{1}{4}} U^\ast \otimes
            \left(U \rho^{-\frac{1}{4}}X_i^*e_{\ell k}\rho^{-\frac{1}{4}} U^\ast \right)^{\op} \\
            &= 
            \sum_{i,k,\ell}\hat{\rho}^{-\frac{1}{4}} (U X_iU^\ast) U e_{k\ell} U^\ast \hat{\rho}^{-\frac{1}{4}}
            \otimes
            \left(\hat{\rho}^{-\frac{1}{4}} (U X_i^* U^\ast) U e_{\ell k} U^\ast \hat{\rho}^{-\frac{1}{4}}\right)^{\op}\\
            &= 
            \sum_{i,k,\ell}\hat{\rho}^{-\frac{1}{4}} (U X_iU^\ast) e_{k\ell} \hat{\rho}^{-\frac{1}{4}} 
            \otimes
            \left(\hat{\rho}^{-\frac{1}{4}} (U X_i^* U^\ast) e_{\ell k} \hat{\rho}^{-\frac{1}{4}}\right)^{\op}.
    \end{align*}
    For the final equality, we use that $(Ue_{k\ell}U^\ast)_{k, \ell}$ provides an alternative set of matrix elements for $M_n$ and the statement from Proposition \ref{pre::prop:S->P} does not depend on the particular matrix elements.
\end{proof}

\subsection{Pauli Spaces of Quantum Graphs on \texorpdfstring{$M_2$}{M2}}
\label{sec::quantum_edge_spaces_and_Pauli_matrices}

To classify the quantum graphs on $(M_2, \psi_q)$ for $q \in (0,1]$ we will use the quantum edge space picture. In view of Definition \ref{pre::def:isomorphism_of_qgraphs}, this means that we need to classify all subspaces $S \subset M_2$ up to the equivalence relation 
\begin{align*}
    S_1 \sim_q S_2
    \quad \vcentcolon \Leftrightarrow \quad 
    \pi(S_1) = S_2 \quad \text{ for some } \quad \pi \in \mathrm{Aut}(M_2, \psi_q).
\end{align*}
Following the approach of \cite{gromada_examples_2022} we use the basis of $M_2$ given by the Pauli matrices
to reformulate this problem in a more accessible way. Indeed, in Proposition \ref{pre::prop:aut_group_of_(M2,psi)} we described the action of the automorphism groups of $(M_2, \psi_q)$ foresightedly in terms of the Pauli matrices. 

First, we discuss the tracial case. For that,  recall $\psi_1 = \tau = 2 \mathrm{Tr}$.

\begin{definition}
    \label{pau::def:pauli_space_of_qgraph}
    For a quantum graph $\mathcal{G}$ on $(M_2, \tau)$ with quantum edge space $S \subset M_2$ set
    \begin{align*}
        V := \left\{\begin{pmatrix}
            x_0 & x_1 & x_2 & x_3
        \end{pmatrix}^t \in \C^4: \; x_0 \sigma_0 + x_1 \sigma_1 +  x_2 \sigma_2 + x_3 \sigma_3 \in S \right\}.
    \end{align*}
    For the purposes of this paper, we call $V$ the \emph{Pauli space} of $\mathcal{G}$.
\end{definition}

\begin{lemma}
\label{pau::lemma:Pauli_spaces_of_isomorphic_quantum_graphs}
    Two quantum graphs $\mathcal{G}_1, \mathcal{G}_2$ on $(M_2, \tau)$ with Pauli spaces $V_1, V_2 \subset \C^4$, respectively, are isomorphic if and only if there is a rotation $R \in \SO(3)$ such that
    \begin{align*}
        \left[\begin{array}{c|c}
            1& 0 \\
            \hline
            0 & R 
            \end{array}\right] V_1 = V_2.
    \end{align*}
    Further, $\mathcal{G}_1$ is loopfree if and only if $V_1 \perp e_1$.
\end{lemma}

\begin{proof}
    The first statement follows directly from Proposition \ref{pre::prop:aut_group_of_(M2,psi)} and Definition \ref{pre::def:isomorphism_of_qgraphs}.
    By Definition \ref{pre::def:(ir)reflxive_qgraphs} the quantum graph $\mathcal{G}_1$ with quantum edge space $S$ is non-loopfree if and only if $\mathrm{I}_2\perp S$ with respect to the KMS-inner product induced by $\tau^{-1}$. One easily checks
    \begin{align*}
        \langle \mathrm{I}_2,\sigma_i\rangle_{\tau^{-1}}^{\KMS}
        =\frac{1}{2}\Tr\left(
        \sigma_i
        \right)=0
    \end{align*}
    for $i=1,2,3$.
    Thus, $\mathcal{G}_1$ is loopfree if and only if $V_1\perp e_1$.
\end{proof}

For nontracial quantum graphs on $(M_2, \psi_q)$ with $q \in (0,1)$ we use a slightly different basis.

\begin{definition}
    \label{pre::def:q_adjusted_Pauli_space}
    Let $\mathcal{G}$ be a quantum graph on $(M_2, \psi_q)$ with quantum edge space $S \subset M_2$. Set
    \begin{align*}
        V^{(q)} := \left\{\begin{pmatrix}
            x_0 & x_1 & x_2 & x_3
        \end{pmatrix}^t \in \C^4: \; x_0 \sigma_0 + x_1 \sigma_1 +  x_2 \sigma_2 + x_3 \sigma_3^{(q)} \in S \right\},
    \end{align*}
    where
    \begin{align*}
        \sigma_3^{(q)} := \begin{pmatrix}
            q^{-2} & 0 \\
            0 & -1
        \end{pmatrix}.
    \end{align*}
    For the purposes of this paper, we call $V^{(q)}$ the \emph{$q$-adjusted Pauli space} of $\mathcal{G}$.
\end{definition}

The advantage of the $q$-adjusted Pauli space is that it is easier to tell from $V^{(q)}$ if the corresponding quantum graph is loopfree.


\begin{lemma}
    \label{pau::lemma:q_adjusted_Pauli_spaces_of_isomorphic_nontracial_qgraphs}
    Two quantum graphs $\mathcal{G}_1, \mathcal{G}_2$ on $(M_2, \psi_q)$ for $q \in (0,1)$ with $q$-adjusted Pauli spaces $V_1^{(q)}, V_2^{(q)} \subset \C^4$, respectively, are isomorphic if and only if there is a rotation $R \in \mathrm{SO}(2)$ such that
    \begin{align*}
        \left[\begin{array}{c|c|c}
            1& 0 & 0 \\
            \hline
            0 & R & 0 \\
            \hline 
            0 & 0 & 1
            \end{array}\right]
            V_1^{(q)} = V_2^{(q)}
            \quad\text{ or equivalently }\quad
            \left[\begin{array}{c|c|c}
            1& 0 & 0 \\
            \hline
            0 & R & 0 \\
            \hline 
            0 & 0 & 1
            \end{array}\right]
            V_1 = V_2.
    \end{align*}
    Further, $\mathcal{G}_1$ is loopfree if and only if $V_1^{(q)} \perp e_1$ holds with respect to the standard inner product on $\C^4$.
\end{lemma}

\begin{proof}
    The first statement follows again from Proposition \ref{pre::prop:aut_group_of_(M2,psi)} and Definition \ref{pre::def:isomorphism_of_qgraphs}. 
    For the final claim, observe
    \begin{align*}
    \langle \mathrm{I}_2, \sigma_3^{(q)} \rangle_{\psi_q^{-1}}^{KMS}
        = \frac{1}{1+q^2} \mathrm{Tr}\left( \begin{pmatrix}
            q & 0 \\
            0 & 1
        \end{pmatrix}
        \begin{pmatrix}
            q^{-2} & 0 \\
            0 & -1
        \end{pmatrix}
        \begin{pmatrix}
            q & 0 \\
            0 & 1
        \end{pmatrix} \right) 
        = 0.
    \end{align*}
    Analogously, one checks $\langle \mathrm{I}_2, \sigma_i \rangle_{\psi_q^{-1}}^{KMS} = 0$ for $i =1, 2$.
    Thus, $\mathcal{G}_1$ is loopfree if and only if $V_1^{(q)} \perp e_1$.
\end{proof}

\begin{remark}
    Note that we consider the spaces $V, V^{(q)} \subset \C^4$ associated to a quantum graph $\mathcal{G}$ on $(M_2, \tau)$, $(M_2, \psi_q)$, respectively, as equipped with the standard inner product of $\C^4$. On the other hand, the quantum edge space $S \subset M_n$ is often equipped with the KMS-inner product induced by $\psi^{-1}$, e.g. in the definition of reflexive/loopfree quantum graphs and in the next section in the definition of complement graphs of $\mathcal{G}$.
\end{remark}

\subsection{Complements of Quantum Graphs}
\label{subsec::complements_of_quantum_graphs}

The aim of this paper is to classify all directed and undirected quantum graphs on $M_2$. For that, it suffices to classify quantum graphs with one or two quantum edges, for the other ones are obtained as complements from these quantum graphs. Let us discuss this in more detail.

In what follows, we describe quantum graphs only by their quantum edge space and do not discuss the related quantum adjacency matrices and quantum edge projections.

\begin{definition}
\label{pre::def:(looopfree)_complement_of_qgraphs}
    Let $\mathcal{G}$ be a quantum graph on the quantum set $(M_n, \psi)$ with quantum edge space $S \subset M_n$. Further, let $S^\perp$ be the orthogonal complement of $S$ with respect to the KMS-inner product induced by $\psi^{-1}$.
    \begin{enumerate}
        \item We call the quantum graph $\mathcal{G}^\perp$ with quantum edge space $S^\perp$ the \emph{complement quantum graph} of $\mathcal{G}$. 
        \item If $\mathcal{G}$ is loopfree, i.e. $\mathrm{I}_n \perp S$, then $\mathcal{G}^\perp$ is reflexive, and we can write 
        $$
        S^\perp = S^\prime \oplus \C \mathrm{I}_n
        $$ 
        as an orthogonal direct sum with respect to the KMS-inner product induced by $\psi^{-1}$. We call the quantum graph $\mathcal{G}^\prime$ with quantum edge space $S^\prime$ the \emph{loopfree complement quantum graph} of $\mathcal{G}$.
    \end{enumerate}
\end{definition}

\begin{remark}
    The complement of the quantum graph $\mathcal{G}$ can equivalently be described in terms of its quantum adjacency matrix. In fact, if $\mathcal{G}$ has quantum adjacency matrix $A: M_n \to M_n$ then $\mathcal{G}^\perp$ has quantum adjacency matrix $J - A$, where $J$ is the quantum adjacency matrix of the complete quantum graph. 
    
    To see this, let $\{X_i\} \cup \{Y_j\}$ be an ONB of $M_n$ with respect to the KMS-inner product induced by $\psi^{-1}$ such that $\{X_i\}$ is an ONB of $S$, where $S$ is the quantum edge space of $\mathcal{G}$. Then, according to Proposition \ref{pre::prop:S->A} the quantum adjacency matrix $A^\perp$ of $\mathcal{G}^\perp$ is given by
    \begin{align*}
        A^\perp: x 
            &\mapsto \sum_j\rho^{-\frac{1}{4}}Y_j\rho^{\frac{1}{4}}x\rho^{\frac{1}{4}}Y_j^*\rho^{-\frac{1}{4}} \\
            &= \left(\sum_i\rho^{-\frac{1}{4}}X_i\rho^{\frac{1}{4}}x\rho^{\frac{1}{4}}X_i^*\rho^{-\frac{1}{4}} + \sum_j\rho^{-\frac{1}{4}}Y_j\rho^{\frac{1}{4}}x\rho^{\frac{1}{4}}Y_j^*\rho^{-\frac{1}{4}} \right) \\
            &\quad - \sum_i\rho^{-\frac{1}{4}}X_i\rho^{\frac{1}{4}}x\rho^{\frac{1}{4}}X_i^*\rho^{-\frac{1}{4}} \\
            &= J(x) - A(x).
    \end{align*}
    Let us mention that such constructions were used by Gromada in \cite[Example 1.18]{gromada_examples_2022}.
\end{remark}

\begin{lemma}
\label{pre::lemma:1-1-correspondence_between_qgraphs_with_k_and_n2-k(-1)_edges}
    With the notation from the previous Definition \ref{pre::def:(looopfree)_complement_of_qgraphs} we have the following.
    \begin{enumerate}
        \item The assignment $\mathcal{G} \mapsto \mathcal{G}^\perp$ defines a 1-1 correspondence between the isomorphism classes of quantum graphs with $k$ quantum edges and quantum graphs with $n^2-k$ quantum edges on $(M_n, \psi)$ for all $k=0, \dots, n^2$.
        \item The assignment $\mathcal{G} \mapsto \mathcal{G}^\prime$ defines a 1-1 correspondence between the isomorphism classes of loopfree quantum graphs with $k$ quantum edges and loopfree quantum graphs with $n^2-k-1$ quantum edges on $(M_n, \psi)$ for all $k = 0, \dots, n^2-1$.
    \end{enumerate}
\end{lemma}

\begin{proof}
    Ad (1). Evidently, if $\mathcal{G}$ has $k$ quantum edges, then $\mathcal{G}^\perp$ has $n^2-k$ quantum edges, and the assignment $\mathcal{G} \mapsto \mathcal{G}^\perp$ is bijective. Thus, it remains to show that the mapping preserves isomorphism classes. For that, assume $\mathcal{G}_1 \cong \mathcal{G}_2$ for two quantum graphs with quantum edge spaces $S_1, S_2 \subset M_n$, respectively. Then there is an automorphism $\pi: M_n \to M_n$ of the quantum set $(M_n, \psi)$ such that $\pi(S_1) = S_2$. It follows directly that $\pi(S_1^\perp) = \pi(S_2^\perp)$ since $\pi$ preserves the KMS-inner product induced by $\psi^{-1}$ (see the proof of Definition \ref{pre::def:isomorphism_of_qgraphs}). Hence, $\pi$ is an isomorphism of $\mathcal{G}_1^\perp$ and $\mathcal{G}_2^\perp$.

    Ad (2). As before the claim about the number of quantum edges is evident. It remains to prove that $\mathcal{G} \mapsto \mathcal{G}^\prime$ preserves isomorphism classes. Let $\mathcal{G}_1 \cong \mathcal{G}_2$ and $\pi$ as above, and assume that the two quantum graphs are loopfree. Since $\pi(\mathrm{I}_n) = \mathrm{I}_n$, one obtains from $\pi(S_1^\perp) = S_2^\perp$ immediately
    \begin{align*}
        \pi(S_1^\prime) \oplus \C \mathrm{I}_n
            = \pi(S_1^\prime \oplus \C \mathrm{I}_n) 
            = \pi(S_1^\perp)
            = S_2^\perp
            = S_2^\prime \oplus \C\mathrm{I}_n,
    \end{align*}
    where all direct sums are orthogonal.
    Consequently, $\pi(S_1^\prime) = S_2^\prime$ and $\mathcal{G}_1^\prime \cong \mathcal{G}_2^\prime$.
\end{proof}

\begin{remark}
    \label{iso::remark:complete_classification}
    In view of the previous lemma, for a complete classification of quantum graphs on $(M_2, \psi)$ it suffices to classify quantum graphs with one quantum edge and non-loopfree quantum graphs with two quantum edges. Indeed, there is only one quantum graph with zero quantum edges, namely the empty quantum graph with quantum edge space $S = \{0\}$. Further, all loopfree quantum graphs with two quantum edges are by Lemma \ref{pre::lemma:1-1-correspondence_between_qgraphs_with_k_and_n2-k(-1)_edges} of the form $\mathcal{G}^\prime$ where $\mathcal{G}$ is a loopfree quantum graph with one quantum edge. Next, all quantum graphs with three quantum edges are of the form $\mathcal{G}^\perp$ for quantum graphs $\mathcal{G}$ with one quantum edge. Finally, there is only one quantum graph with four quantum edges, namely the complete quantum graph.
\end{remark}

\section{Classification of Lines in \texorpdfstring{$\C^2$}{C2} and \texorpdfstring{$\C^3$}{C3}}
\label{sec::classification_of_lines}

In this section we prepare later work  by classifying lines in $\C^2$ and $\C^3$ up to particular orthogonal transformations. 

\subsection{Lines in \texorpdfstring{$\C^2$}{C2} and the M\"obius transformation}
\label{lin::subsec:lines_in_C2}

In this section we classify the lines in $\C^2$ up to a rotation $R \in \mathrm{SO}(2)$ and up to an orthogonal transformation $R \in \mathrm{O}(2)$, respectively.
The central tool for this is the Möbius transformation which defines an action of $\mathrm{GL}_2(\C)$ on $\hat{\C} \vcentcolon = \C \cup \{\infty\}$ via 
\begin{align*}
    \begin{pmatrix}
        a & b \\
        c & d
    \end{pmatrix}
    \cdot z = \frac{az+b}{cz+d}. 
\end{align*}
We will use the following two facts:
\begin{enumerate}
    \item To any nonzero vector $u = \begin{pmatrix}
        z_1 & z_2
    \end{pmatrix}^t \in \C^2$ we may associate the number 
    $$
    \varphi(u) := \frac{z_1}{z_2}
    $$
    in the complex sphere $\hat{\C}$. This number completely describes the line $\spann \{u\}$, and for any  two nonzero vectors $u_1, u_2 \in \C^2$ one has 
    \begin{align*}
        u_1 \in \spann\{u_2\} 
        \quad \Leftrightarrow \quad
        \varphi(u_1) = \varphi(u_2).
    \end{align*}
    \item For any matrix $X \in \mathrm{GL}_2(\C)$ and vector $u \in \C^2$ one has
    \begin{align*}
        \varphi(X u) = X \cdot \varphi(u).
    \end{align*}
\end{enumerate}

\begin{lemma}
    \label{lin::lemma:lines_in_C2_via_Möbius_trafo}
    Let $v \in \C^2 \setminus \{0\}$.
    \begin{enumerate}
        \item There is a unique $\beta \in [-1, 1]$ such that
        \begin{align*}
            Rv \in \spann\left\{ \begin{pmatrix}
                1 \\ i \beta
            \end{pmatrix}\right\}
        \end{align*}
        holds for some $R \in \SO(2)$.
        \item There is a unique $\beta \in [0,1]$ such that
        \begin{align*}
            Rv \in \spann\left\{ \begin{pmatrix}
                1 \\ i \beta
            \end{pmatrix}\right\}
        \end{align*}
        holds for some $R \in \mathrm{O}(2)$.
    \end{enumerate}
\end{lemma}

\begin{proof}
    Due to the properties of the Möbius transformation discussed above, we have
    \begin{align*}
        Rv \in \spann \{ v_\beta \}
        \quad \Leftrightarrow\quad 
        R \cdot \varphi(v) = \varphi(v_\beta)
    \end{align*}
    for all $R \in \mathrm{O}(2)$ and $v_\beta := \begin{pmatrix}
        1 & i \beta
    \end{pmatrix}^t\in\C^2$.
    Hence, to prove the two statements we need to consider the orbits of the action of $\mathrm{SO}(2)$ and $\mathrm{O}(2)$, respectively, on $\hat{\C}$ given by the M\"obius transformation.

    Ad (1). Any special orthogonal matrix $R \in \SO(2)$ is of the form
    \begin{align*}
        R = \begin{pmatrix}
            a   & b \\
            -b & a
        \end{pmatrix}
    \end{align*}
    for some $a+ib \in S^1$ with $a, b \in \R$. For the matrix
    \begin{align*}
        P := \begin{pmatrix}
            1 & -i \\
            1 & i
        \end{pmatrix}
    \end{align*}
    one has
    \begin{align*}
        P \begin{pmatrix}
            a   & b \\
            -b & a
        \end{pmatrix}
        P^{-1} 
        = \begin{pmatrix}
            a +ib  & 0 \\
            0 & a-ib
        \end{pmatrix}.
    \end{align*}
    Writing $a+ib = e^{it}$, it follows for all $z \in \hat{\C}$
    \begin{align*}
        \left(P \begin{pmatrix}
            a   & b \\
            -b & a
        \end{pmatrix}
        P^{-1}\right)
        \cdot z = \frac{e^{it}z}{e^{-it}} = e^{2it}z.
    \end{align*}
    Using that, we get for any two $z_1, z_2 \in \hat{\C}$
    \begin{align*}
         &R \cdot z_1 = z_2 \quad \text{ for some } R \in \SO(2)\\
        \Leftrightarrow \qquad & (P R P^{-1}) \cdot(P \cdot z_1) = P \cdot z_2  \quad \text{ for some } R \in \SO(2) \\
        \Leftrightarrow \qquad & e^{2it} (P \cdot z_1) = P \cdot z_2 \quad \text{ for some } t \in \R\\
        \Leftrightarrow \qquad & |P \cdot z_1| = |P \cdot z_2|,
    \end{align*}
    where we use the convention $|\infty| = \infty$. Thus, the orbits of $\SO(2) \curvearrowright \hat{\C}$ are exactly
    \begin{align*}
        O_r = \{P^{-1} \cdot z: |z| = r\} \quad \text{ with } r \in [0, \infty].
    \end{align*}
    A system of representatives is given by the numbers $-i \beta^{-1}$ with $\beta \in [-1, 1]$. Indeed, we have
    \begin{align*}
        - i \beta^{-1} = P^{-1} \cdot (P \cdot (-i \beta^{-1})) 
        \quad \text{ with } \quad 
        P \cdot (-i \beta^{-1}) = \frac{-i \beta^{-1} - i}{-i \beta^{-1}+i} = \frac{1+\beta}{1-\beta},
    \end{align*}
    and the latter term defines a bijective map from $[-1,1]$ to $[0, \infty]$. Since for the vector $v_\beta := \begin{pmatrix}
        1 & i \beta
    \end{pmatrix}^t \in \C^2$ we have $\varphi(v_\beta) = \frac{1}{i \beta} = - i \beta^{-1}$ the statement follows.

    Ad (2). Any matrix in $\mathrm{O}(2)$ is of the form
    \begin{align*}
        \begin{pmatrix}
            -1 & 0 \\
            0 & 1
        \end{pmatrix}^\varepsilon R
        \quad \text{ with } \quad 
        \varepsilon \in \{0,1\}, R \in \SO(2).
    \end{align*}
    In order to classify the orbits of the action $\mathrm{O}(2) \curvearrowright \hat{\C}$ given by the M\"obius transformation we, therefore, need only find out which orbits $O_r$ of $\SO(2) \curvearrowright \hat{\C}$ are linked by the matrix $\mathrm{diag}(-1, 1)$. For any $z \in \hat{\C}$ one has 
    \begin{align*}
        \left(P \begin{pmatrix}
            -1 & 0 \\
            0  & 1
        \end{pmatrix} P^{-1}\right) \cdot z
            = \begin{pmatrix}
                0 & -1 \\
                -1 & 0
            \end{pmatrix}
            \cdot z = \frac{-1}{-z} = \frac{1}{z}.
    \end{align*}
    Hence, for any $z_1, z_2 \in \hat{\C}$ it is
    \begin{align*}
        &\mathrm{diag}(-1, 1) \cdot z_1 = z_2 \\
        \Leftrightarrow \quad & \left(P \mathrm{diag}(-1, 1) P^{-1}\right) \cdot (P \cdot z_1) = P \cdot z_2 \\
        \Leftrightarrow \quad & (P \cdot z_1)^{-1} = P \cdot z_2,
    \end{align*}
    and for any $z \in \hat{\C}$ the numbers $P^{-1} \cdot z$ and $P^{-1} \cdot z^{-1}$ are in the same orbit of $\mathrm{O}(2) \curvearrowright \hat{\C}$. Consequently, the orbits of this action are exactly
    \begin{align*}
        O^\prime_r := \{P^{-1} \cdot z: |z| \in \{r, r^{-1}\}\} \quad \text{ with } r \in [1,\infty].
    \end{align*}
    A system of representatives is given by $\{-i \beta^{-1}: \beta \in [0,1]\}$ since for $\beta \in [0,1]$ the term
    \begin{align*}
        P \cdot (-i \beta^{-1}) = \frac{1+\beta}{1-\beta}
    \end{align*}
    ranges over $[1, \infty]$. The statement follows in the same way as above.
\end{proof}

The previous proof has a very visual interpretation: Every line $\spann \{u\}$ with $u \in \C^2$ corresponds to a point on the Riemann sphere. After applying the transformation $P$, the action of $\SO(2) \cong S^1$ corresponds to spinning this sphere around the axis through the points $0$ and $\infty$. The group $\mathrm{O}(2)$ is obtained from $\SO(2)$ by adding the generator $\mathrm{diag}(-1, 1)$ whose action on the sphere swaps the upper and lower half sphere.

\subsection{Vectors in \texorpdfstring{$\C^2$}{C2} and \texorpdfstring{$\C^3$}{C3} up to a real rotation}
%
In this section, we prepare the lemmas which will be crucial to our later classification of quantum graphs. Building on the previous section, we find canonical spanning vectors of lines in $\C^3$ and $\C^2$ up to a real rotation.
For later purposes 
we also list all real rotations which preserve these canonical vectors.

\begin{lemma}
    \label{lin::lemma:existence_of_index_beta_for_lines_in_C3}
        For every non-vanishing vector $v \in \C^3$ there exists some $\lambda \in \C$ and $R \in \SO(3)$ such that one has
        \begin{align*}
            \lambda R v = \begin{pmatrix}
                1 \\ i \beta \\ 0
            \end{pmatrix}
        \end{align*}
        for some $\beta \in [0,1]$. Similarly, for every non-vanishing $v \in \C^2$ there exists some $\lambda \in \C$ and $R \in \SO(2)$ such that
        \begin{align*}
            \lambda R v = \begin{pmatrix}
                1 \\ i \beta
            \end{pmatrix}
        \end{align*}
        holds for some $\beta \in [-1,1]$.
\end{lemma}

\begin{proof}
    Let $v = x + iy$ for real vectors $x, y \in \R^3$. Clearly, there is a rotation $R_1 \in \SO(3)$ such that $R_1 x \in \spann\{e_1\}$. Let $R_1 y =: \begin{pmatrix}
        y_1 & y_2 & y_3
    \end{pmatrix}^t$. Then there is a rotation $R_2 \in \SO(2)$ which sends $\begin{pmatrix}
        y_2 & y_3
    \end{pmatrix}^t$ into $\spann\left\{ \begin{pmatrix}
        1 & 0
    \end{pmatrix}^t\right\}$. Consequently, we have
    \begin{align*}
        \left[\begin{array}{c|c}
            1 & 0 \\ \hline
            0 & R_2
        \end{array}\right] R_1 v = \begin{pmatrix}
            a + i b \\ ic \\ 0
        \end{pmatrix}
    \end{align*}
    for some real numbers $a, b, c \in \R$. Finally, by Lemma \ref{lin::lemma:lines_in_C2_via_Möbius_trafo}(2) there exists some orthogonal matrix $R_3 \in \mathrm{O}(2)$ and $\beta \in [0,1]$ such that
    \begin{align*}
        R_3 \begin{pmatrix}
            a + i b \\ ic
        \end{pmatrix}
        \in \spann\left\{ \begin{pmatrix}
            1 \\ i \beta
        \end{pmatrix} \right\}.
    \end{align*}
    By choosing $\kappa \in \{1, -1\}$ and $\lambda \in \C$ suitably we obtain 
    \begin{align*}
        \left[\begin{array}{c|c}
            R_3 & 0 \\ \hline
            0 & \kappa
        \end{array}\right] \in \SO(3)
        \quad \text{ and } \quad 
        \lambda \left[\begin{array}{c|c}
            R_3 & 0 \\ \hline
            0 & \kappa
        \end{array}\right] \left[\begin{array}{c|c}
            1 & 0 \\ \hline
            0 & R_2
        \end{array}\right] R_1 v 
        = \begin{pmatrix}
            1 \\ i \beta \\ 0
        \end{pmatrix}.
    \end{align*}
    The last statement follows directly from Lemma \ref{lin::lemma:lines_in_C2_via_Möbius_trafo}.
\end{proof}

\begin{lemma}
    \label{lin::lemma:uniqueness_of_index_beta_for_lines_in_C3}
    For $\beta, \beta^\prime \in [0,1], \lambda \in \C$ and $R \in \SO(3)$ we have
    \begin{align}
    \label{dir::eq:beta_beta_prime_are_index_of_the_same_line}
        \lambda R \begin{pmatrix}
            1 \\ i \beta \\ 0
        \end{pmatrix}
        = \begin{pmatrix}
            1 \\ i \beta^\prime \\ 0
        \end{pmatrix}
    \end{align}
    if and only if $\beta = \beta^\prime$ and one of the following is true. 
    \begin{enumerate}[label=(\alph*)]
        \item
        \label{lin::item:beta_beta_prime_are_index_of_the_same_line:a}
        $\beta = 1, \lambda = x + iy \in S^1$ ($x,y\in\R$) and $\displaystyle R = \begin{pmatrix}
            x & -y & 0 \\
            y & x & 0 \\
            0 & 0 & 1
        \end{pmatrix}$,
        \item 
        \label{lin::item:beta_beta_prime_are_index_of_the_same_line:b}
        $\beta \in (0,1), \lambda = \pm 1$ and $\displaystyle R = \begin{pmatrix}
            \lambda & 0 & 0 \\
            0 & \lambda & 0 \\
            0 & 0 & 1
        \end{pmatrix}$,
        \item 
        \label{lin::item:beta_beta_prime_are_index_of_the_same_line:c}
        $\beta = 0, \lambda = \pm 1$ and $\displaystyle R = \left[\begin{array}{c|c}
            \lambda & 0 \\ \hline
            0 & Q
        \end{array}\right]$
        for some $Q \in \mathrm{O}(2)$ with $\mathrm{det}(Q) = \lambda$.
    \end{enumerate}
\end{lemma}
    
\begin{proof}
    First, it is a simple calculation to check that \ref{lin::item:beta_beta_prime_are_index_of_the_same_line:a}, \ref{lin::item:beta_beta_prime_are_index_of_the_same_line:b} and \ref{lin::item:beta_beta_prime_are_index_of_the_same_line:c} entail (\ref{dir::eq:beta_beta_prime_are_index_of_the_same_line}). Write 
    \begin{align*}
        R = \begin{pmatrix}
            a & b & * \\
            c & d & * \\
            e & f & *
        \end{pmatrix}
    \end{align*}
    with $a, b, c, d, e, f \in \R$. Then (\ref{dir::eq:beta_beta_prime_are_index_of_the_same_line}) reads
    \begin{align} \label{dir::eq:lambda_R_(1, i beta, 0)^t}
        \lambda R \begin{pmatrix}
            1 \\ i \beta \\ 0
        \end{pmatrix}
        = \lambda \begin{pmatrix}
            a + i b \beta \\ c + i d \beta \\ e + if \beta
        \end{pmatrix}
        = \begin{pmatrix}
            1 \\ i \beta^\prime \\ 0
        \end{pmatrix}.
    \end{align}
    For $\beta = 0$, one easily checks that (\ref{dir::eq:beta_beta_prime_are_index_of_the_same_line}) holds if and only if \ref{lin::item:beta_beta_prime_are_index_of_the_same_line:c} from the statement is true.
    Thus, we assume from now on $\beta \neq 0$. Then the equality in the third component in (\ref{dir::eq:lambda_R_(1, i beta, 0)^t}) entails $e = f = 0$. Since the matrix $R$ is orthogonal with determinant $1$, $R$ must be of the form
    \begin{align*}
        R = \begin{pmatrix}
             a & b & 0 \\
            \mp b & \pm a & 0 \\
            0 & 0 & \pm 1
        \end{pmatrix}
    \end{align*}
    with $a^2 + b^2 = 1$. Then (\ref{dir::eq:beta_beta_prime_are_index_of_the_same_line}) becomes
    \begin{align*}
        \lambda R \begin{pmatrix}
            1 \\ i \beta \\ 0
        \end{pmatrix}
        = \lambda \begin{pmatrix}
            a + i b \beta \\ \mp b + (\pm i a \beta) \\ 0
        \end{pmatrix}
        = \begin{pmatrix}
            1 \\ i \beta^\prime \\ 0
        \end{pmatrix}.
    \end{align*}
    Equality in the first component yields 
    \begin{align*}
        \lambda = \frac{1}{a + i b \beta}.
    \end{align*}
    Then the second component gives us
    \begin{align*}
        \frac{\mp b + (\pm i a \beta)}{a + i b \beta}
        = \frac{(\mp b + (\pm i a \beta))(a - i b \beta)}{a^2 + b^2 \beta^2}
        &= \frac{\mp ab  + (\pm ab \beta^2) + i (\pm a^2 \beta + (\pm b^2 \beta))}{a^2 + b^2 \beta^2} \\
        &= i \beta^\prime.
    \end{align*}
    Looking at the real and imaginary part of this equation separately, we arrive at
    \begin{align}
        \label{dir::eq::real_part_vanishes} 
        0 &= \mp ab  + (\pm ab \beta^2) = \pm ab (\beta^2 - 1), \\
        \label{dir::eq::imaginary_part_equals_beta_prime}
        \pm \beta^\prime &= \frac{\ a^2 \beta +  b^2 \beta}{a^2 + b^2 \beta^2} = \frac{\beta}{a^2 + b^2 \beta^2} .
    \end{align}
    We do a case distinction to discuss the solutions $a, b, \beta, \beta^\prime$ to these equations.
    \begin{enumerate}
        \item If $\beta = 1$, then (\ref{dir::eq::real_part_vanishes}) is satisfied by all $a, b \in \R$ with $a^2 + b^2 = 1$. The equation (\ref{dir::eq::imaginary_part_equals_beta_prime}) is equivalent to
        \begin{align*}
            \pm \beta^\prime = \beta = 1.
        \end{align*}
        Since $\beta, \beta^\prime \in [0,1]$ the sign in the above equation must be positive, and we obtain $\beta^\prime = \beta$ as well as
        \begin{align*}
            R = \begin{pmatrix}
                a & b & 0 \\
                -b & a & 0 \\
                0 & 0 & 1
            \end{pmatrix}.
        \end{align*}
        Further, we have
        $
            \lambda = \frac{1}{a+ib} = a-ib,
        $
        and this proves \ref{lin::item:beta_beta_prime_are_index_of_the_same_line:a}.
        \item If $\beta \neq 1$ and $a \neq 0$, then (\ref{dir::eq::real_part_vanishes}) entails $b = 0$. Consequently, we must have $a \in \{-1,1\}$ since $a+ib \in S^1$ is orthogonal, and (\ref{dir::eq::imaginary_part_equals_beta_prime}) tells us
        \begin{align*}
            \pm \beta^\prime = \beta.
        \end{align*}
        As before, the sign in this equation must be positive, and we get $\beta^\prime = \beta$ as well as
        \begin{align*}
            R = \begin{pmatrix}
                -1 & 0 & 0 \\
                0 & -1 & 0 \\
                0 & 0 & 1
            \end{pmatrix}
            \quad \text{ or } \quad 
            R = \begin{pmatrix}
                1 & 0 & 0 \\
                0 & 1 & 0 \\
                0 & 0 & 1
            \end{pmatrix}
        \end{align*}
        depending on whether $a = 1$ or $a = -1$.
        Further, $\lambda = \frac{1}{a}$ and thus \ref{lin::item:beta_beta_prime_are_index_of_the_same_line:b} from the statement holds.
        \item If $\beta \neq 1$ and $a =0$, then $b \in \{-1, 1\}$ and (\ref{dir::eq::imaginary_part_equals_beta_prime}) yields
        \begin{align*}
            \pm \beta^\prime = \frac{1}{\beta}.
        \end{align*}
        (Recall that we assume $\beta \neq 0$.) This contradicts the assumption that both $\beta$ and $\beta^\prime$ are in $[0,1]$, and therefore this case is not applicable.
    \end{enumerate}
\end{proof}

For the classification of nontracial quantum graphs we will need the following very similar lemma.

\begin{lemma}
    \label{lin::lemma:uniqueness_of_lambda_in_C2}
    For $\beta, \beta^\prime \in [-1,1]$, $\lambda \in \C$ and $R \in \SO(2)$ we have
    \begin{align}
    \label{dir::eq:beta_beta_prime_are_index_of_the_same_line_in_C2}
        \lambda R \begin{pmatrix}
            1 \\ i \beta
        \end{pmatrix}
        = \begin{pmatrix}
            1 \\ i \beta^\prime
        \end{pmatrix}
    \end{align}
    if and only if $\beta = \beta^\prime$ and one of the following is true.
    \begin{enumerate}[label=(\alph*)]
        \item\label{lin::item:beta_beta_prime_are_index_of_the_same_line_in_C2:a} $\beta = \pm 1, \lambda = x + iy \in S^1$ ($x, y \in \R$) and $\displaystyle R = \begin{pmatrix}
            x & -\beta y \\
            \beta y & x \\
        \end{pmatrix}$,
        \item\label{lin::item:beta_beta_prime_are_index_of_the_same_line_in_C2:b} $\beta \in (-1,1), \lambda = \pm 1$ and $\displaystyle R = \lambda \mathrm{I}_2$.
    \end{enumerate}
\end{lemma}

\begin{proof}
    The proof is very similar to the previous lemma. First, one checks that \ref{lin::item:beta_beta_prime_are_index_of_the_same_line_in_C2:a} and \ref{lin::item:beta_beta_prime_are_index_of_the_same_line_in_C2:b} entail (\ref{dir::eq:beta_beta_prime_are_index_of_the_same_line_in_C2}). Now, if $\beta = 0$, then one immediately observes (\ref{dir::eq:beta_beta_prime_are_index_of_the_same_line_in_C2}) implies \ref{lin::item:beta_beta_prime_are_index_of_the_same_line_in_C2:b}. To prove this implication for general $\beta$, write 
    \begin{align*}
        R := \begin{pmatrix}
             a & b \\
            -b & a
        \end{pmatrix}
    \end{align*}
    with $a + ib \in S^1$. Similarly as in the previous proof, the equation (\ref{dir::eq:beta_beta_prime_are_index_of_the_same_line_in_C2}) is equivalent to 
    $$
    \lambda = \frac{1}{a+ib \beta}
    $$ 
    and
    \begin{align}
        \label{dir::eq::real_part_vanishes_in_C2} 
        0 &= - ab  + (ab \beta^2) = ab (\beta^2 - 1), \\
        \label{dir::eq::imaginary_part_equals_beta_prime_in_C2}
        \beta^\prime &= \frac{\ a^2 \beta +  b^2 \beta}{a^2 + b^2 \beta^2} = \frac{\beta}{a^2 + b^2 \beta^2}.
    \end{align}
    Let us make a case distinction.
    \begin{enumerate}
        \item If $\beta = \pm 1$, then (\ref{dir::eq::real_part_vanishes_in_C2}) is satisfied by all $a+ib \in S^1$. The equation (\ref{dir::eq::imaginary_part_equals_beta_prime_in_C2}) yields
        \begin{align*}
            \beta^\prime = \beta
        \end{align*}
        and further we have $\lambda = \frac{1}{a + ib \beta} = a - ib\beta$.
        Thus, we obtain \ref{lin::item:beta_beta_prime_are_index_of_the_same_line_in_C2:a}.
        \item If $|\beta| < 1$ and $a \neq 0$, then (\ref{dir::eq::real_part_vanishes_in_C2}) entails $b = 0$. Consequently, we must have $a \in \{-1,1\}$ since $a+ib \in S^1$ and (\ref{dir::eq::imaginary_part_equals_beta_prime_in_C2}) tells us
        \begin{align*}
            \beta^\prime = \beta.
        \end{align*}
        Thus, we obtain \ref{lin::item:beta_beta_prime_are_index_of_the_same_line_in_C2:b} with $\lambda = \frac{1}{a+ib\beta} = a$.
        \item If $0 \neq |\beta| < 1$ and $a =0$, then we must have $b= \pm 1$. Hence, (\ref{dir::eq::imaginary_part_equals_beta_prime_in_C2}) yields
        \begin{align*}
            \beta^\prime = \frac{1}{\beta}.
        \end{align*}
        This contradicts the assumption that both $\beta$ and $\beta^\prime$ are in $[-1,1]$. Therefore this case is not applicable.
    \end{enumerate}
\end{proof}

\section{Tracial Quantum Graphs with One Quantum Edge}
\label{sec::t1e}

In this section, we classify all quantum graphs on $(M_2, \tau)=(M_2, 2 \mathrm{Tr})$ with one quantum edge. As discussed in Section \ref{sec::quantum_edge_spaces_and_Pauli_matrices}, for that we need to classify all lines $V \subset \C^4$ up to a real rotation of the last three components. With the help of Lemmas \ref{lin::lemma:existence_of_index_beta_for_lines_in_C3} and \ref{lin::lemma:uniqueness_of_index_beta_for_lines_in_C3} one obtains the following result.

\begin{lemma}
    \label{t1e::lemma:classification_of_lines_V}
    For every line in $V\subset\C^4$ exactly one of the following is true. 
    \begin{enumerate}[label=(\Alph*)]
        \item\label{enumA::trace1edge:index_of_line_in_C4} It is $V= \spann\{e_1\}=\vcentcolon V^{(1A)}$.
        \item\label{enumB::trace1edge:index_of_line_in_C4} There are $\alpha \in \C$ and $\beta \in [0,1]$  such that
        \begin{align*}
        \left[\begin{array}{c|c}
                1& 0 \\
                \hline
                0 & R 
                \end{array}\right]
            V =  \spann\left\{ \begin{pmatrix}
                    \alpha \\ 1 \\ i \beta \\ 0
                \end{pmatrix} \right\}
                =\vcentcolon
                V^{(1B)}_{\alpha,\beta}
        \end{align*}
        holds for some $R \in \SO(3)$ and 
        \begin{align}
        \label{tra::eq:J^(1B)}\tag{$\ast$}
        \begin{aligned}
                \beta = 1 \quad &\implies \quad \alpha \in [0, \infty), \\
                \beta \in [0,1) \quad &\implies \quad \arg(\alpha) \in [0, \pi).
        \end{aligned}
        \end{align}
        Denote the index set $
        J^{(1B)}\vcentcolon=
        \{(\alpha,\beta)\vcentcolon\text{\eqref{tra::eq:J^(1B)} is fulfilled}\}$.
    \end{enumerate}
    Further, the numbers $\alpha$ and $\beta$ in \ref{enumB::trace1edge:index_of_line_in_C4} are unique.
\end{lemma}

\begin{proof}
    If $V = \spann\{e_1\}$, then \ref{enumA::trace1edge:index_of_line_in_C4} is obviously true. Further, we have 
    \begin{align*}
         \left[\begin{array}{c|c}
        1 & 0 \\
        \hline
        0 & R
        \end{array}\right] \spann\{e_1\} = \spann\{e_1\}
    \end{align*}
    for all $R \in \SO(3)$ and therefore \ref{enumB::trace1edge:index_of_line_in_C4} is false. It remains to consider the case where $V \neq \spann\{e_1\}$. Then there is some vector
    \begin{align*}
        v = \begin{pmatrix}
            v_0 \\ v_1 \\ v_2 \\ v_3
        \end{pmatrix}
        \in V 
        \quad \text{ with }
        v^\prime := \begin{pmatrix}
            v_1 \\ v_2 \\ v_3
        \end{pmatrix} \neq 0.
    \end{align*}
    Note that $v$ spans $V$ since the latter is a one-dimensional vector space. By Lemma \ref{lin::lemma:existence_of_index_beta_for_lines_in_C3} there is a $\beta \in [0,1]$ with
    \begin{align*}
        \begin{pmatrix}
            1 & i \beta & 0
        \end{pmatrix}^t
        = \lambda R v^\prime,
    \end{align*}
    for some $\lambda \in \C$ and $R \in \SO(3)$. Let $\alpha \in \C$ be given by 
    \begin{align*}
        \begin{pmatrix}
            \alpha \\ 1 \\ i \beta \\ 0
        \end{pmatrix}
        = \lambda \left[\begin{array}{c|c}
            1 & 0 \\
            \hline
            0 & R
            \end{array}\right] v.
    \end{align*}
    We discuss the uniqueness of $\alpha \in \C$ and $\beta \in [0,1]$. For that, assume that we have
    \begin{align*}
        \begin{pmatrix}
            \alpha \\ 1 \\ i \beta \\ 0
        \end{pmatrix}
        =\lambda'
        \left[\begin{array}{c|c}
            1 & 0 \\
            \hline
            0 & R'
            \end{array}\right]
        \begin{pmatrix}
            \alpha' \\ 1 \\ i \beta' \\ 0
        \end{pmatrix}
    \end{align*}
    for some $\lambda'\in\C$, $R'\in\SO(3)$, $\alpha'\in\C$ and $\beta'\in[0,1]$.
    By Lemma \ref{lin::lemma:uniqueness_of_index_beta_for_lines_in_C3} the latter equation holds if and only if one of the following is true.
    \begin{enumerate}[label=(\alph*)]
        \item $\beta=1$, $\lambda'=x+iy\in S^1$ ($x,y\in\R$) and 
        $\displaystyle R=\begin{pmatrix}x&-y&0\\y&x&0\\0&0&1\end{pmatrix}$.
        \item $\beta\in(0,1)$, $\lambda' = \pm 1$  and $\displaystyle R = \begin{pmatrix}
            \lambda^\prime & 0 & 0 \\
            0 & \lambda^\prime & 0 \\
            0 & 0 & 1
        \end{pmatrix}$,
        \item $\beta=0$, $\lambda^\prime = \pm 1$ and $\displaystyle R = \left[\begin{array}{c|c}
            \lambda^\prime & 0 \\ \hline
            0 & Q
        \end{array}\right]$ for some $Q \in \mathrm{O}(2)$ with $\mathrm{det}(Q) = \lambda^\prime$.
    \end{enumerate}
    In all cases we have $\beta=\beta'$ as well as $\alpha=\lambda'\alpha'$. Thus, in the case $\beta=1$, the number $\alpha$ is unique up to a multiplication with a scalar from $S^1$; while in the latter cases $\alpha$ is unique up to sign.

\end{proof}

Using the previous lemma, we can classify the tracial quantum graphs with one quantum edge. 

\comments{
\begin{thm}
    \label{dir::thm:all_directed_loopfree_qgraphs_with_one_edge_on_M_2}
    For every $(\alpha, \beta) \in J^{(1B)}$ let $\mathcal{G}_{\alpha, \beta}^{(1B)}$ be the quantum graph with Pauli space
    \begin{align*}
        V_{\alpha, \beta}^{(1B)} = \spann \left\{ \begin{pmatrix}
            \alpha \\ 1 \\i \beta \\ 0
        \end{pmatrix} \right\}.
    \end{align*}
    Any quantum graph $\mathcal{G}$ on $(M_2,\tau)$ with one quantum edge is isomorphic to either the trivial quantum graph $\mathcal{G}^{(1A)}$ (see Example \ref{pre::example:empty_trivial_complete_qgraph})  with self-loops \todo{Nina an Björn 05??: Wieso werden hier die self-loops hervorgehoben?} or exactly one of the quantum graphs $\mathcal{G}_{\alpha,\beta}^{(1B)}$ with $(\alpha,\beta)\in J^{(1B)}$.
\end{thm}
}

\begin{thm}
    \label{t1e::thm:classification_of_1te_qgraphs}
    Let $\mathcal{G}^{(1A)}$ and $\mathcal{G}_{\alpha, \beta}^{(1B)}$ be the quantum graphs with Pauli spaces $V^{(1A)}$ and $V_{\alpha, \beta}^{(q, 1B)}$, respectively.
    Any quantum graph $\mathcal{G}$ on $(M_2,\tau)$ with exactly one quantum edge is isomorphic to either $\mathcal{G}^{(1A)}$, the trivial quantum graph from Example \ref{pre::example:empty_trivial_complete_qgraph}, or exactly one of the quantum graphs $\mathcal{G}_{\alpha,\beta}^{(1B)}$ with $(\alpha,\beta)\in J^{(1B)}$.
\end{thm}

\begin{proof}
    Let $\mathcal{G}$ be an arbitrary quantum graph on $(M_2, \tau)$ with one quantum edge. By Lemma \ref{t1e::lemma:classification_of_lines_V} its Pauli space $V$ is either $\spann\{e_1\}$ or of the form
    \begin{align*}
        V = \left[\begin{array}{c|c}
            1& 0 \\
            \hline
            0 & R 
            \end{array}\right] 
            V_{\alpha, \beta}^{(1B)}
    \end{align*}
    for some $R \in \SO(3)$ and unique $(\alpha,\beta)\in J^{(1B)}$.
     In the first case, $\mathcal{G}$ is the trivial graph with self-loops and one quantum edge. Otherwise, Lemma $\ref{pau::lemma:Pauli_spaces_of_isomorphic_quantum_graphs}$ yields that $\mathcal{G}$ is isomorphic to $\mathcal{G}_{\alpha, \beta}^{(1B)}$. 
\end{proof}

In Example \ref{pre::example:properties} we already discussed the properties of the trivial graph $\mathcal{G}^{(1A)}$. Let us now also investigate the quantum graphs $\mathcal{G}_{\alpha, \beta}^{(1B)}$ with $(\alpha,\beta)\in J^{(2)}$.

\begin{proposition}
    \label{t1e::prop:properties_of_G_alpha,beta^(1B)}
    Let $(\alpha,\beta)\in J^{(1B)}$ and denote $\beta_+\vcentcolon=1+\beta$ and $\beta_-=1-\beta$.
    The quantum graph $\mathcal{G}_{\alpha, \beta}^{(1B)}$ has the quantum edge space
    \begin{align*}
        S_{\alpha,\beta}^{(1B)}=
        \spann\left\{
        \begin{pmatrix}
            \alpha&\beta_+\\ \beta_-&\alpha
        \end{pmatrix}
        \right\}
    \end{align*}
    and the quantum adjacency matrix
    \begin{align*}
        A_{\alpha,\beta}^{(1B)}=
        c^{-1}
        \begin{pmatrix}
            |\alpha|^2&\alpha\beta_+&\overline{\alpha}\beta_+&\beta_+^2\\
            \alpha\beta_-&|\alpha|^2&\beta_-\beta_+&\overline{\alpha}\beta_+\\
            \overline{\alpha}\beta_-&\beta_-\beta_+&|\alpha|^2&\alpha\beta_+\\
            \beta_-^2&\overline{\alpha}\beta_-&\alpha\beta_-&|\alpha|^2
        \end{pmatrix},
    \end{align*}
    where $c\vcentcolon=|\alpha|^2+1+\beta^2$. Its spectrum is
    \begin{align*}
        \frac{1}{|\alpha|^2+1+\beta^2}
        \Big\{
            & \left|\alpha + \sqrt{1-\beta^2}\right|^2,
            \left(\alpha - \sqrt{1-\beta^2}\right)\left(\overline{\alpha} + \sqrt{1-\beta^2}\right), \\
            & \left| \alpha - \sqrt{1-\beta^2} \right|^2, \left(\alpha + \sqrt{1-\beta^2}\right)\left(\overline{\alpha} - \sqrt{1-\beta^2}\right)
        \Big\}.
    \end{align*}
    The quantum graph $\mathcal{G}_{\alpha,\beta}^{(1B)}$ is
    \begin{itemize}
        \item (GNS- respectively KMS-) undirected if and only if $\alpha\in\R$ and $\beta=0$.
        \item never reflexive,
        \item loopfree if and only if $\alpha=0$.
    \end{itemize}
\end{proposition}

\begin{proof}
    In this proof, let us omit the indices $\alpha, \beta$ and $(1B)$ for the sake of brevity, i.e. we denote $\mathcal{G}_{\alpha, \beta}^{(q)}$ simply by $\mathcal{G}$.
    Using the definition of the Pauli space of a quantum graph from Section \ref{sec::quantum_edge_spaces_and_Pauli_matrices}, the quantum edge space of $\mathcal{G}$ is given by 
    \begin{align*}
        \spann\left\{
        \alpha \sigma_0 +\sigma_1+i\beta\sigma_2
        \right\}=
        \spann\left\{
        \begin{pmatrix}
            \alpha&1+\beta\\1-\beta&\alpha
        \end{pmatrix}
        \right\}.
    \end{align*}
    Set 
    \begin{align*}
        T\vcentcolon
        = \begin{pmatrix}
            \alpha&\beta_+\\\beta_-&\alpha
        \end{pmatrix}
    \end{align*}
    and recall that $\tau = 2 \mathrm{Tr} = \mathrm{Tr}(\rho \, \cdot)$ with
    \begin{align*}
        \rho = \begin{pmatrix}
            2 & 0 \\
            0 & 2
        \end{pmatrix}.
    \end{align*}
    Thus, $\tau^{-1} = 1/2\, \mathrm{Tr}$ and the KMS-inner product induced by $\tau^{-1}$ is given by 
    \begin{align*}
        \langle x, y \rangle_{\tau^{-1}}^{\KMS} = \frac{1}{2} \mathrm{Tr}(x^\ast y).
    \end{align*}
    The norm of $T$ with respect to the KMS-inner product induced by $\tau^{-1}$ is
    \begin{align*}
        \|T\|^2
        =\left\langle
        T,T
        \right\rangle_{\tau^{-1}}^{\KMS}
        &=\frac{1}{2} \Tr\left(
        \begin{pmatrix}
            \overline{\alpha}&\beta_-\\\beta_+&\overline{\alpha}
        \end{pmatrix}
        \begin{pmatrix}
            \alpha&\beta_+\\\beta_-&\alpha
        \end{pmatrix}
        \right) 
        = |\alpha|^2 + 1 + \beta^2
        = c.
    \end{align*}
    Thus, $\{c^{-1/2}T\}$ is an orthonormal basis of $S$, and, after identifying the quantum adjacency matrix $A\vcentcolon M_2\to M_2$ with a matrix $A \in M_4$ according to Proposition \ref{pre::prop:S->A} and Remark \ref{pre::rem:adjacency_matrix_A_as_Kronecker_product}, we have
    \begin{align*}
        A = c^{-1} \left( T \otimes \overline{T} \right)
        =c^{-1}\begin{pmatrix}
            \alpha&\beta_+\\\beta_-&\alpha
        \end{pmatrix}
        \otimes
        \begin{pmatrix}
            \overline{\alpha}&\beta_+\\\beta_-&\overline{\alpha}
        \end{pmatrix}.
    \end{align*}
    A simple calculation of the Kronecker product yields the desired expression for $A$.
    It is easy to check that the spectrum of $T$ consists of the two numbers
    \begin{align*}
        \alpha + \sqrt{1-\beta^2}
        \quad \text{ and }\quad 
        \alpha - \sqrt{1-\beta^2}.
    \end{align*}
    Therefore, the spectrum of $A = c^{-1}\left( T \otimes \overline{T} \right)$ is the set
    \begin{align*}
        c^{-1} \left\{\alpha + \sqrt{1-\beta^2}, \alpha - \sqrt{1-\beta^2}\right\} \cdot \left\{\overline{\alpha} + \sqrt{1-\beta^2}, \overline{\alpha} - \sqrt{1-\beta^2}\right\},
    \end{align*}
    where the elements of the set on the right-hand side are pairwise multiplied with the elements of the set on the left-hand side.

    Recall that, in the tracial setting, a quantum graph is GNS-undirected if and only if it is KMS-undirected if and only if the quantum edge space is self-adjoint. 
    Since $S = \spann\{T\}$, this condition is equivalent to
    \begin{align*}
        \begin{pmatrix}
            \overline{\alpha}&\beta_-\\
            \beta_+&\overline{\alpha}
        \end{pmatrix}
        =\lambda
        \begin{pmatrix}
            \alpha&\beta_+\\
            \beta_-&\alpha
        \end{pmatrix}
    \end{align*}
    for some $\lambda\in\C$.
    Since $\beta_-,\beta_+\in[0,2]$ the equations $\beta_+=\lambda\beta_-$ and $\beta_-=\lambda\beta_+$ are fulfilled if and only if $\lambda=1$ and $\beta=0$, and thus, $\alpha\in\R$.
    Next, the graph $\mathcal{G}$ is reflexive if and only if $\mathrm{I}_2\in S$. Since $\beta_+\neq0$ for all $\beta\in[0,1]$ this is not true and thus, the quantum graph $\mathcal{G}$ is not reflexive. Finally, $\mathcal{G}$ is loopfree if and only if $\langle\mathrm{I}_2,T\rangle_{\tau^{-1}}^{\KMS}=\alpha=0$.
\end{proof}

\begin{remark}
    As a special case we retain from Theorem \ref{t1e::thm:classification_of_1te_qgraphs} a classification of all undirected quantum graphs on $(M_2, \tau)$. Up to isomorphism there are exactly two (GNS-/KMS-)undirected quantum graphs on $(M_2, \tau)$: the trivial quantum graph $\mathcal{G}^{(1A)}$ and the quantum graph $\mathcal{G}^{(1B)}_{0,0}$ with quantum adjacency matrix
    \begin{align*}
        A_{0,0}^{(1B)}=
        \begin{pmatrix}
            0 &0 &0 &1 \\
            0 &0 &1 &0 \\
            0 &1 &0 &0 \\
            1 &0 &0 &0
        \end{pmatrix},
    \end{align*}
    see Proposition \ref{t1e::prop:properties_of_G_alpha,beta^(1B)}. One can checks that this aligns perfectly with Matsuda's and Gromada's classification of tracial undirected quantum graphs on $M_2$, see \cite[Theorem 3.1]{matsuda_classification_2022} and \cite[Theorem 3.11]{gromada_examples_2022}.
\end{remark}

\section{Non-Loopfree Tracial Quantum Graphs with Two Quantum Edges}
\label{sec::t2e}

In this section, we classify all non-loopfree quantum graphs on $(M_2, \tau)=(M_2, 2 \mathrm{Tr})$ with two quantum edges. In view of Remark \ref{iso::remark:complete_classification}, together with the classification of quantum graphs on $(M_2, \tau)$ with one quantum edge this yields a complete classificaton of quantum graphs on $(M_2, \tau)$. 

Non-loopfree quantum graphs with two quantum edges are given by two-dimensional Pauli spaces $V \subset \C^4$ which are not orthogonal to $\spann\{e_1\}$. Similarly to the previous section, we use Lemmas \ref{lin::lemma:existence_of_index_beta_for_lines_in_C3} and \ref{lin::lemma:uniqueness_of_index_beta_for_lines_in_C3} to classify these spaces up to a real rotation of the last three components.

\begin{lemma}
    \label{t2e::lemma:classification_of_planes_V}
    Let $V\subset\C^4$ be a plane with $V\not\perp e_1$. Then there are unique numbers $\beta\in[0,1]$, $\gamma,\delta\in\C$ such that
    \begin{align*}
        \left[\begin{array}{c|c}
                1& 0 \\
                \hline
                0 & R
                \end{array}\right]
        V=
        \spann \left\{
            \begin{pmatrix}
                0 \\ 1 \\ i \beta \\ 0
            \end{pmatrix},
            \begin{pmatrix}
                1 \\ i \beta \gamma \\ \gamma \\ \delta
            \end{pmatrix}
            \right\}
            =\vcentcolon V^{(2)}_{\beta,\gamma,\delta}
    \end{align*}
    holds for some $R\in\SO(2)$, and
    \begin{align}
        \label{tra::eq:J^(2)}\tag{$\ast$}
        \begin{aligned}
            \beta=0, \frac{\gamma}{\delta}\neq\pm i &\implies \arg(\gamma)\in[0,\pi),\arg(\delta)=\arg(\gamma)+\pi,|\delta|\leq|\gamma|,\\ 
            \beta=0, \frac{\gamma}{\delta}=\pm i &\implies \gamma\in[0,\infty),\arg(\delta)=\arg(\gamma)+\pi,|\delta|\leq\gamma, \\
            \beta\in(0,1) &\implies \arg(\gamma)\in[0,\pi), \\
            \beta=1&\implies\gamma\in[0,\infty).
        \end{aligned}
    \end{align}
    Denote the index set $
        J^{(2)}\vcentcolon=
        \{(\beta,\gamma,\delta)\vcentcolon\text{\eqref{tra::eq:J^(2)} is fulfilled}\}$.
\end{lemma}

\begin{proof}
         Since $V$ is not orthogonal to $e_1$ there is an (up to a scalar factor) unique vector $v \in \C^3$ such that
        \begin{align*}
            v =\vcentcolon \begin{pmatrix}
                0 \\ v_1 \\ v_2 \\ v_3
            \end{pmatrix}
            \in V.
        \end{align*}
     Indeed, if there were two linear independent such vectors $v, v' \in \C^3$, then $\spann\{v,v^\prime\}$ would be a two-dimensional subspace of $V$. Since $\mathrm{dim}(V)=2$, it would follow that $V$ is orthogonal to $e_1$ -- contradicting the assumption. 
    Since $v$ is orthogonal to $e_1$, by Lemma \ref{lin::lemma:existence_of_index_beta_for_lines_in_C3} there exists a number $\beta \in [0,1]$ such that
    \begin{align*}
            \begin{pmatrix}
                0 \\ 1 \\ i \beta \\ 0
            \end{pmatrix}
            = \lambda
            \left[\begin{array}{c|c}
            1& 0 \\
            \hline
            0 & R
            \end{array}\right] 
            v
    \end{align*}
    holds for some $\lambda\in\C$ and an $R\in\SO(3)$. 
    Let $w\in V$ be another vector that is orthogonal to $v$ and such that $V=\spann\{v,w\}$. Then $w$ has non-vanishing first component. Without loss of generality, we may assume that its first component is $w_1=1$. Using $v\perp w$ one can check that there are numbers $\gamma,\delta\in\C$ such that
    \begin{align*}
        \begin{pmatrix}
            1 \\ i \beta \gamma \\ \gamma \\ \delta
        \end{pmatrix}
        = 
        \left[\begin{array}{c|c}
            1& 0 \\
            \hline
            0 & R
            \end{array}\right] 
            w.
    \end{align*}
    It remains to investigate uniqueness of $\beta,\gamma,\delta$. 
    Assume that 
    \begin{align*}
        \spann \left\{
        \begin{pmatrix}
            0 \\ 1 \\ i \beta \\ 0
        \end{pmatrix},
        \begin{pmatrix}
            1 \\ i \beta \gamma \\ \gamma \\ \delta
        \end{pmatrix}
        \right\}
        =\left[\begin{array}{c|c}
            1& 0 \\
            \hline
            0 & R
            \end{array}\right]
        \spann \left\{
        \begin{pmatrix}
            0 \\ 1 \\ i \beta' \\ 0
        \end{pmatrix},
        \begin{pmatrix}
            1 \\ i \beta' \gamma' \\ \gamma' \\ \delta'
        \end{pmatrix}
        \right\}
    \end{align*}
    holds for some $R\in\SO(3)$, $\beta'\in[0,1]$ and $\gamma',\delta'\in \C$.
    Looking at the first components of the spanning vectors one checks that this is equivalent to 
    \begin{align*}
        \begin{pmatrix}
            0 \\ 1 \\ i \beta \\ 0
        \end{pmatrix}
        =\lambda^\prime
        \left[\begin{array}{c|c}
            1& 0 \\
            \hline
            0 & R
            \end{array}\right]
        \begin{pmatrix}
            0 \\ 1 \\ i \beta' \\ 0
        \end{pmatrix}
        \quad\text{and}\quad
        \begin{pmatrix}
            1 \\ i \beta \gamma \\ \gamma \\ \delta
        \end{pmatrix}
        =
        \left[\begin{array}{c|c}
            1& 0 \\
            \hline
            0 & R
            \end{array}\right]
        \begin{pmatrix}
            1 \\ i \beta' \gamma' \\ \gamma' \\ \delta'
        \end{pmatrix}
    \end{align*}
    for some $\lambda^\prime\in\C$.
    The first equation and Lemma \ref{lin::lemma:uniqueness_of_index_beta_for_lines_in_C3} entail $\beta'=\beta$. 
    It remains 
    to investigate uniqueness of the numbers $\gamma,\delta\in\C$ with
    \begin{align}
    \label{eq::2equations_in_proof_of_Lemma_2edges_tracialcase}
        \begin{pmatrix}
                1 \\ i \beta \\ 0
        \end{pmatrix}
        =\lambda^\prime R
            \begin{pmatrix}
                1 \\ i \beta \\ 0
            \end{pmatrix}
        \quad\text{ and }\quad
        \begin{pmatrix}
                i \beta \gamma \\ \gamma \\ \delta
            \end{pmatrix}
        = R
            \begin{pmatrix}
                i \beta \gamma' \\ \gamma' \\ \delta'
            \end{pmatrix}.
    \end{align}
    By Lemma \ref{lin::lemma:uniqueness_of_index_beta_for_lines_in_C3} the first equality holds if and only if one of the following is true.
    \begin{enumerate}[label=(\alph*)]
        \item $\beta = 1, \lambda^\prime = x + iy \in S^1$ ($x,y\in\R$) and $\displaystyle R = \begin{pmatrix}
            x & -y & 0 \\
            y & x & 0 \\
            0 & 0 & 1
        \end{pmatrix}$. In this case, one checks that the second equation in (\ref{eq::2equations_in_proof_of_Lemma_2edges_tracialcase}) is equivalent to $\gamma=\lambda^\prime\gamma'$ and $\delta=\delta'$. Thus, the pair $(\gamma, \delta)$ can be chosen uniquely such that $\gamma \geq 0$.
        \item 
        $\beta \in (0,1), \lambda^\prime = \pm 1$ and $\displaystyle R = \begin{pmatrix}
            \lambda^\prime & 0 & 0 \\
            0 & \lambda^\prime & 0 \\
            0 & 0 & 1
        \end{pmatrix}$.
        In this case the second equation is equivalent to $\gamma=\lambda^\prime\gamma'$ and $\delta=\delta'$. Thus, the pair $(\gamma, \delta)$ can be chosen uniquely such that $\arg(\gamma)\in[0,\pi)$.
        \item 
        $\beta = 0, \lambda^\prime = \pm 1$ and $\displaystyle R = \left[\begin{array}{c|c}
            \lambda^\prime  & 0 \\ \hline
            0 & Q
        \end{array}\right]
        \in\SO(3)$
        for some $Q \in \mathrm{O}(2)$ with $\mathrm{det}(Q) = \lambda^\prime$. In this case, the second equality is equivalent to
        \begin{align*}
            \begin{pmatrix}
                \gamma \\ \delta
            \end{pmatrix}
            =Q
            \begin{pmatrix}
                \gamma^\prime \\ \delta'
            \end{pmatrix}.
        \end{align*}
        By Lemma \ref{lin::lemma:lines_in_C2_via_Möbius_trafo} we can choose $Q \in \mathrm{O}(2)$ such that
        \begin{align*}
            \kappa \begin{pmatrix}
                1 \\ i \eta
            \end{pmatrix}
            =Q
            \begin{pmatrix}
                \gamma' \\ \delta'
            \end{pmatrix}
        \end{align*}
        holds for some $\kappa \in \C$ and $\eta \in [0,1]$, where
        $\eta$ is unique by Lemma \ref{lin::lemma:uniqueness_of_lambda_in_C2}. By the same lemma, the number $\kappa$ is unique up to scalar multiplication from $S^1$ if $\eta=1$ and unique up to a sign if $\eta < 1$. The first case applies if and only if $\gamma/\delta=\pm i$. Indeed, the map $\varphi$ from Section \ref{lin::subsec:lines_in_C2} that sends a vector from $\C^2$ to the quotient of its first and second component satisfies   
        \begin{align*}
            \varphi\left( \begin{pmatrix}
                1 \\ i
            \end{pmatrix} \right) = - i
        \end{align*}
        and an inspection of the Möbius transform from Section \ref{lin::subsec:lines_in_C2} shows
        \begin{align*}
            \varphi(u) = - i
            \quad \Leftrightarrow \quad 
            \exists T \in \mathrm{O}(2): \; \varphi(Tu) = \pm i
        \end{align*}
        for all $u \in \C^2 \setminus \{0\}$.
        It follows that the pair $(\gamma, \delta)$ can be chosen uniquely such that
        \begin{align*}
            \mathrm{arg}(\gamma) \in [0, \pi)\,, \mathrm{arg}(\delta) = \mathrm{arg}(\gamma) + \pi \text{ and } |\delta| \leq |\gamma|
        \end{align*}
        and additionally $\gamma \geq 0$ if $\gamma/\delta= \pm i$.
    \end{enumerate}
\end{proof}

Using the previous lemma, we can classify the tracial non-loopfree quantum graphs with two quantum edges.

\begin{thm}
    \label{t2e::thm:classification_of_t2e_qgraphs}
    Let $\mathcal{G}_{\beta,\gamma,\delta}^{(2)}$ be the quantum graph with Pauli space $V_{\beta,\gamma,\delta}^{(2)}$.
    Any non-loopfree quantum graph $\mathcal{G}$ on $(M_2,\tau)$ with exactly two quantum edges is isomorphic to exactly one of the quantum graphs $\mathcal{G}_{\beta,\gamma,\delta}^{(2)}$ where $(\beta,\gamma,\delta)\in J^{(2)}$.
\end{thm}

\begin{proof}
    Let $\mathcal{G}$ be a quantum graph on $(M_2, \tau)$ with Pauli space $V \subset \C^4$. As $\mathcal{G}$ is non-loopfree, we have  $V \not \perp e_1$  by Lemma \ref{pau::lemma:Pauli_spaces_of_isomorphic_quantum_graphs}. By Lemma \ref{t2e::lemma:classification_of_planes_V} there exists a unique triple $(\beta, \gamma, \delta) \in J^{(2)}$ with
    \begin{align*}
        V = \left[\begin{array}{c|c}
            1& 0 \\
            \hline
            0 & R 
            \end{array}\right] 
            V_{\beta, \gamma, \delta}^{(2)}
    \end{align*}
    for some $R \in \SO(3)$. The statement follows from Lemma \ref{pau::lemma:Pauli_spaces_of_isomorphic_quantum_graphs}.
\end{proof}

According to the last theorem, we can describe all non-loopfree quantum graphs with two quantum edges by $V_{\beta,\gamma,\delta}^{(2)}$ where $(\beta,\gamma,\delta)\in J^{(2)}$.

\begin{proposition}
\label{t2e::prop:properties_of_G_alpha,beta,delta^(2)}
    Let $(\beta,\gamma,\delta)\in J^{(2)}$.
    Denote $\delta_+\vcentcolon=1+\delta\in\C$, $\delta_-\vcentcolon=1-\delta\in\C$, $\beta_+\vcentcolon=1+\beta$ and $\beta_-\vcentcolon=1-\beta$.
    The quantum graph $\mathcal{G}_{\beta,\gamma,\delta}^{(2)}$ has the quantum edge space
    \begin{align*}
        S_{\beta,\gamma,\delta}^{(2)}=\spann\left\{
            \begin{pmatrix}
                0&\beta_+\\
                \beta_-&0
            \end{pmatrix},
            \begin{pmatrix}
                \delta_+&-i\gamma\beta_-\\
                i\gamma\beta_+&\delta_-
            \end{pmatrix}
            \right\}
    \end{align*}
    and the quantum adjacency matrix
    \begin{align*}
        A_{\beta,\gamma,\delta}^{(2)}=
        &c_1^{-1}
        \begin{pmatrix}
            0&0&0&\beta_+^2\\
            0&0&\beta_-\beta_+&0\\
            0&\beta_-\beta_+&0&0\\
            \beta_-^2&0&0&0
        \end{pmatrix}\\
        &+
        c_2^{-1}
            \begin{pmatrix}
                |\delta_+|^2&i\overline{\gamma}\beta_-\delta_+&-i\gamma\beta_-\overline{\delta}_+&|\gamma|^2\beta_-^2\\
                -i\overline{\gamma}\beta_+\delta_+&\delta_+\overline{\delta}_-&-|\gamma|^2\beta_-\beta_+&-i\gamma\beta_-\overline{\delta}_-\\
                i\gamma\beta_+\overline{\delta}_+&-|\gamma|^2\beta_-\beta_+&\overline{\delta}_+\delta_-&i\overline{\gamma}\beta_-\delta_-\\
                |\gamma|^2\beta_+^2&i\gamma\beta_+\overline{\delta}_-&-i\overline{\gamma}\beta_+\delta_-&|\delta_-|^2
            \end{pmatrix}
    \end{align*}
    where
        $c_1\vcentcolon=1+\beta^2
        \text{ and }
        c_2\vcentcolon=1+|\delta|^2+|\gamma|^2(1+\beta^2)$.
    The quantum graph $\mathcal{G}_{\beta,\gamma,\delta}^{(2)}$ is
    \begin{itemize}
        \item (GNS- respectively KMS-) undirected if and only if $\beta=0$ and $\gamma,\delta\in\R$, 
        \item reflexive if and only if $\delta=\gamma=0$,
        \item never loopfree.
    \end{itemize}
\end{proposition}

\begin{proof}
    The proof is similar to the proof of Proposition \ref{t1e::prop:properties_of_G_alpha,beta^(1B)}.
    To simplify notation, we omit the indices in the proof.
    Using the definition of the Pauli space of a quantum graph from Section \ref{sec::quantum_edge_spaces_and_Pauli_matrices}, the quantum edge space of $\mathcal{G}$ is
    \begin{align*}
        S \vcentcolon
        =
        \spann\left\{
            \begin{pmatrix}
                0&\beta_+\\
                \beta_-&0
            \end{pmatrix},
            \begin{pmatrix}
                \delta_+&-i\gamma\beta_-\\
                i\gamma\beta_+&\delta_-
            \end{pmatrix}
            \right\}.
    \end{align*}
    Denote
    \begin{align*}
        T_1 := \begin{pmatrix}
                0&\beta_+\\
                \beta_-&0
            \end{pmatrix}
        \quad \text{ and } \quad 
        T_2 := \begin{pmatrix}
                \delta_+&-i\gamma\beta_-\\
                i\gamma\beta_+&\delta_-
            \end{pmatrix}.
    \end{align*}
    To calculate the adjacency matrix, we normalize $T_1$ and $T_2$ with respect to the KMS-inner product and get
    \begin{align*}
        \|T_1\|^2=
        \frac{1}{2}\Tr\left(
        \begin{pmatrix}
            0&1-\beta\\
            1+\beta&0
        \end{pmatrix}
        \begin{pmatrix}
            0&1+\beta\\
            1-\beta&0
        \end{pmatrix}
        \right)
        =1+\beta^2
        = c_1
    \end{align*}
    and
    \begin{align*}
        \|T_2\|^2
        &=
        \frac{1}{2}\Tr\left(
        \begin{pmatrix}
            1+\overline{\delta}&-i\overline{\gamma}(1+\beta)\\
            i\overline{\gamma}(1-\beta)&1-\overline{\delta}
        \end{pmatrix}
        \begin{pmatrix}
            1+\delta&-i\gamma(1-\beta)\\
            i\gamma(1+\beta)&1-\delta
        \end{pmatrix}
        \right)\\
        &=1+|\delta|^2+|\gamma|^2(1+\beta^2)
         = c_2.
    \end{align*}
    With Proposition \ref{pre::prop:S->A} and Remark \ref{pre::rem:adjacency_matrix_A_as_Kronecker_product}, we get the adjacency matrix
    \begin{align*}
        A \vcentcolon= 
        c_1^{-1}
        \begin{pmatrix}
            0&\beta_+\\
            \beta_-&0
        \end{pmatrix}
        \otimes
        \begin{pmatrix}
            0&\beta_+\\
            \beta_-&0
        \end{pmatrix}
        +c_2^{-1}
        \begin{pmatrix}
            \delta_+&-i\gamma\beta_-\\
            i\gamma\beta_+&\delta_-
        \end{pmatrix}
        \otimes
        \begin{pmatrix}
            \overline{\delta}_+&i\overline{\gamma}\beta_-\\
            -i\overline{\gamma}\beta_+&\overline{\delta}_-
        \end{pmatrix},
    \end{align*}
    and the statement follows by calculating the Kronecker product.

    Recall that the quantum graph $\mathcal{G}$ is undirected if and only if $\mathcal{G}$ is KMS-undirected if and only if $\mathcal{G}$ is GNS-undirected. By Definition \ref{pre::def:GNS-KMS-undirected_quantum_graphs}, this is equivalent to $S = S^\ast$. Observe that $T_1^\ast \in S$ is equivalent to $T_1^\ast=\mu_1T_1+\nu_1 T_2$ for some $\mu_1,\nu_1\in\C$. Switching to the Pauli space this holds if and only if
    \begin{align*}
        \begin{pmatrix}
            0\\1\\-i\beta\\0
        \end{pmatrix}
        =
        \mu_1\begin{pmatrix}
            0\\1\\i\beta\\0
        \end{pmatrix}
        +\nu_1\begin{pmatrix}
            1\\i\beta\gamma\\ \gamma\\ \delta
        \end{pmatrix}
    \end{align*}
    which is equivalent to $\mu_1=1$, $\nu_1=0$ and $\beta=0$. Analogously one has $T_2^\ast=\mu_2T_1+\nu_2 T_2$ for some $\mu_2,\nu_2\in\C$ if and only if
    \begin{align*}
        \begin{pmatrix}
            1\\0\\ \overline{\gamma}\\ \overline{\delta}
        \end{pmatrix}
        =\mu_2\begin{pmatrix}
            0\\1\\0\\0
        \end{pmatrix}
        +\nu_2\begin{pmatrix}
            1\\0\\ \gamma\\ \delta
        \end{pmatrix}
    \end{align*}
    if and only if $\mu_2=0$, $\nu_2=1$ and $\gamma,\delta\in\R$. Thus, $\mathcal{G}$ is undirected if and only if $\beta=0$ and $\gamma, \delta \in \R$.
    
    Further, recall that $\mathcal{G}$ is reflexive if and only if $\mathrm{I}_2\in S$. Equivalently, the equation
    \begin{align*}
        \begin{pmatrix}
            1\\0\\0\\0
        \end{pmatrix}
        =\mu\begin{pmatrix}
            0\\1\\i\beta\\0
        \end{pmatrix}
        +\nu\begin{pmatrix}
            1\\i\beta\gamma\\ \gamma\\ \delta
        \end{pmatrix}
    \end{align*}
    is fulfilled for some $\mu,\nu\in\C$ if and only if
    $\nu=1$, $\mu=0$ and $\gamma=\delta=0$ since $T_1$ and $T_2$ are orthogonal.
    By assumption, the graph $\mathcal{G}$ is never loopfree.
\end{proof}

\begin{remark}
    In view of the previous proposition and Theorem \ref{t2e::thm:classification_of_t2e_qgraphs} any reflexive quantum graph with two quantum edges on $(M_2, \tau)$ is isomorphic to exactly one of the quantum graphs $\mathcal{G}_{\beta, 0, 0}^{(2)}$ with $\beta \in [0,1]$. These quantum graphs have adjacency matrix
    \begin{align*}
        A_{\beta,0,0}^{(2)}
        =
        \begin{pmatrix}
            1&0&0&\frac{1+\beta}{1-\beta}\\
            0&1&1&0\\
            0&1&1&0\\
            \frac{1-\beta}{1+\beta}&0&0&1
        \end{pmatrix}.
    \end{align*}
    The only undirected one of these quantum graphs is $\mathcal{G}_{0,0,0}^{(2)}$ with adjacency matrix
    \begin{align*}
        A_{0,0,0}^{(2)}
        =
        \begin{pmatrix}
            1&0&0&1\\
            0&1&1&0\\
            0&1&1&0\\
            1&0&0&1
        \end{pmatrix}.
    \end{align*}
\end{remark}

\section{Nontracial Quantum Graphs with One Quantum Edge}
\label{sec::n1e}

In this section, we classify the nontracial quantum graphs on $(M_2, \psi_q)$ for $q \in (0,1)$. This is analogous to the classification of quantum graphs with one quantum edge on $(M_2, \tau)$ from Section \ref{sec::t1e}. In particular, we obtain undirected and/or loopfree quantum graphs as special cases and we recover Matsuda's result that there is exactly one GNS-undirected loopfree quantum graph with one quantum edge on $(M_2, \psi_q)$ \cite{matsuda_classification_2022}. We also observe that there exists a larger number of loopfree, KMS-undirected quantum graphs with one quantum edge on $(M_2, \psi_q)$. In fact, these quantum graphs include the GNS-undirected loopfree quantum graphs as well as a family of quantum graphs indexed by a single nonnegative number $\gamma \geq 0$, see Remark \ref{ntg::rmk:KMS-undirected_qgraphs_with_one_edge_on_(M2,psi_q)}.

As usual, we start with a classification of Pauli spaces. Different to the previous calculations, however, we aim at classifying the $q$-adjusted Pauli spaces from Definition \ref{pre::def:q_adjusted_Pauli_space}. This makes no difference for the classification problem but allows for a better handling of loopfree quantum graphs later on. Following Section \ref{sec::quantum_edge_spaces_and_Pauli_matrices}, we are interested in lines, i.e. one-dimensional subspaces, in $\C^4$ up to a rotation of the middle two components.

\begin{lemma}
\label{ntg::lemma:Pauli_spaces_for_qgraphs_with_one_edge_on_M2,psi_q}
    For every line in $V\in\C^4$ exactly one of the following is true.
    \begin{enumerate}[label=(\Alph*)]
        \item\label{enumA::trace2edges:index_of_plane_in_C4} It is $\displaystyle V  = \spann \{e_1\}=\vcentcolon V^{(q,1A)}$.
        \item\label{enumB::trace2edges:index_of_plane_in_C4} There is an $\alpha \in \C$ such that 
        $\displaystyle V 
        =\spann\left\{ \begin{pmatrix}
            \alpha \\ 0 \\ 0 \\ 1
        \end{pmatrix} \right\}
        =\vcentcolon V^{(q,1B)}_{\alpha}.$
        \item\label{enumC::trace2edges:index_of_plane_in_C4} There are $\alpha, \gamma \in \C$ and $\beta \in [-1, 1]$ such that
        \begin{align*}
        \left[\begin{array}{c|c|c}
            1 & 0 & 0\\
            \hline
            0 & R & 0 \\
            \hline
            0 & 0 & 1
            \end{array}\right]
        \displaystyle V
        =  \spann\left\{ \begin{pmatrix}
            \alpha \\ 1 \\ i \beta \\ \gamma
        \end{pmatrix} \right\}
        =\vcentcolon V^{(q,1C)}_{\alpha,\beta,\gamma}
        \end{align*}
        holds for some $R \in \SO(2)$ and
        \begin{align}
        \label{tra::eq:J^(1,qC)}\tag{$\ast$}
        \begin{aligned}
                \beta = \pm 1, \alpha \neq 0 \quad &\implies \quad \alpha \in (0, \infty), \\
                \beta = \pm 1, \alpha = 0 \quad &\implies \quad \gamma \in [0, \infty), \\
                \beta \in (-1,1), \alpha \neq 0 \quad &\implies \quad \mathrm{arg}(\alpha) \in [0, \pi), \\
                \beta \in (-1,1), \alpha = 0 \quad &\implies \quad \mathrm{arg}(\gamma) \in [0, \pi).
            \end{aligned}
        \end{align}
        Denote the index set $
        J^{(q,1C)}\vcentcolon=
        \{(\alpha,\beta,\gamma)\vcentcolon\text{\eqref{tra::eq:J^(1,qC)} is fulfilled}\}$.
    \end{enumerate}
    Further, the numbers $\alpha,\beta,\gamma$ in \ref{enumB::trace2edges:index_of_plane_in_C4} and/or \ref{enumC::trace2edges:index_of_plane_in_C4}, respectively, are unique.
\end{lemma}

\begin{proof}
    If $V$ is orthogonal to $\spann\{e_2, e_3\}$, then it is not hard to check that either \ref{enumA::trace2edges:index_of_plane_in_C4} or \ref{enumB::trace2edges:index_of_plane_in_C4} is true, and the number $\alpha$ in \ref{enumB::trace2edges:index_of_plane_in_C4} is unique. Otherwise, we have $V = \spann \{v\}$ for some vector $v \in \C^4$ such that its middle two components are not both vanishing. By Lemma \ref{lin::lemma:lines_in_C2_via_Möbius_trafo}, there is a unique $\beta \in [-1, 1]$ such that
    \begin{align}
    \label{ntg::eq:canonical_vector_with_alpha_beta_gamma_for_one_edge}
        \begin{pmatrix}
            \alpha \\ 1 \\ i \beta \\ \gamma
        \end{pmatrix}
        =
        \lambda \left[\begin{array}{c|c|c}
            1 & 0 & 0\\
            \hline
            0 & R & 0 \\
            \hline
            0 & 0 & 1
        \end{array}\right]
        v
    \end{align}
    holds for some $R \in \SO(2), \lambda \in \C$ and $\alpha, \gamma \in \C$. To investigate uniqueness of the numbers $\alpha, \gamma$, assume that we have
    \begin{align} 
    \label{ntg::eq:uniqueness_of_alpha_gamma}
        \begin{pmatrix}
            \alpha \\ 1 \\ i \beta \\ \gamma
        \end{pmatrix}
        = \lambda^\prime \left[\begin{array}{c|c|c}
            1 & 0 & 0\\
            \hline
            0 & R^\prime & 0 \\
            \hline
            0 & 0 & 1
            \end{array}\right] \begin{pmatrix}
                \alpha^\prime \\ 1 \\ i \beta' \\ \gamma^\prime
            \end{pmatrix}
    \end{align}
    for some $\lambda^\prime \in \C, R^\prime \in \SO(2)$ as well as $\alpha^\prime, \gamma^\prime \in \C$, $\beta'\in[0,1]$. 
    Lemma \ref{lin::lemma:uniqueness_of_lambda_in_C2} tells us that this holds if and only if one of the following is true.
    \begin{enumerate}[label=(\alph*)]
        \item $\beta = \beta^\prime = \pm 1$, $\lambda^\prime = x + iy \in S^1$ and $\displaystyle R^\prime = \begin{pmatrix}
            x & -\beta y \\
            \beta y & x \\
        \end{pmatrix}$.
        \item $\beta = \beta^\prime \in (-1,1)$, $\lambda^\prime = \pm1$ and $\displaystyle R^\prime = \lambda'\mathrm{I}_2$.
    \end{enumerate}
    In any event, it is $\beta = \beta^\prime$ as well as $\alpha = \lambda^\prime \alpha^\prime$ and $\gamma = \lambda^\prime \gamma^\prime$. Thus, in the case $\beta = \pm 1$, the numbers $\alpha$ and $\gamma$ in (\ref{ntg::eq:canonical_vector_with_alpha_beta_gamma_for_one_edge}) are unique up to multiplication with a scalar from $S^1$; while in the latter case $\alpha$ and $\gamma$ are unique up to a sign. The statement follows immediately. Indeed, by multiplying with a suitable scalar, in the first situation one can choose the pair $(\alpha, \gamma)$ canonically such that either $\alpha \in (0, \infty)$ or $\gamma \in [0, \infty)$ if $\alpha = 0$. In the second situation one can choose the pair canonically such that $\mathrm{arg}(\alpha) \in [0, \pi)$ or $\mathrm{arg}(\gamma) \in [0, \pi)$ if $\alpha = 0$.
\end{proof}

Using the previous lemma, we can classify the nontracial quantum graphs on $(M_2, \psi_q)$ with one quantum edge.

\comments{
\begin{thm}
    \label{ntg::thm:classification_nontracial_qgraphs_with_one_edge}
         Let $\mathcal{G}_{\alpha, \beta, \gamma}^{(q, 1C)}$ be the quantum graph with $q$-adjusted Pauli space 
    \begin{align*}
        V_{\alpha, \beta, \gamma}^{(q, 1C)} = \spann\left\{ \begin{pmatrix}
            \alpha \\ 1 \\ i \beta \\ \gamma
        \end{pmatrix} \right\}.
    \end{align*}
    Every quantum graph $\mathcal{G}$ on $(M_2, \psi_q)$ with exactly one quantum edge is either the trivial quantum graph $\mathcal{G}^{(q,1A)}$ or isomorphic to exactly one of the quantum graphs $\mathcal{G}^{(q,1B)}_\alpha$, $\mathcal{G}_{\alpha, \beta, \gamma}^{(q,1C)}$, where $\alpha \in \C$, and $(\alpha, \beta, \gamma) \in J^{(q,1C)}$, respectively.
\end{thm}
}

\begin{thm}
    \label{ntg::thm:classification_nontracial_qgraphs_with_one_edge}
    Let $\mathcal{G}^{(q, 1A)}$, $\mathcal{G}_{\alpha}^{(q, 1B)}$ and $\mathcal{G}_{\alpha, \beta, \gamma}^{(q, 1C)}$ be the quantum graphs on $(M_2, \psi_q)$ with $q$-adjusted Pauli spaces $V^{(q, 1A)}$, $V_{\alpha}^{(q, 1B)}$ and $V_{\alpha, \beta, \gamma}^{(q, 1C)}$, respectively.
    Every quantum graph $\mathcal{G}$ on $(M_2, \psi_q)$ with exactly one quantum edge is either the trivial quantum graph $\mathcal{G}^{(q,1A)}$ or isomorphic to exactly one of the quantum graphs $\mathcal{G}^{(q,1B)}_\alpha$, $\mathcal{G}_{\alpha, \beta, \gamma}^{(q,1C)}$ where $\alpha \in \C$ and $(\alpha, \beta, \gamma) \in J^{(q,1C)}$.
\end{thm}

\begin{proof}
    Let $\mathcal{G}$ be an arbitrary quantum graph on $(M_2, \psi_q)$ with $q$-adjusted Pauli space $V \subset \C^4$. By Lemma \ref{pau::lemma:q_adjusted_Pauli_spaces_of_isomorphic_nontracial_qgraphs}, another quantum graph $\mathcal{G}^\prime$ on $(M_2, \psi_q)$ with $q$-adjusted Pauli space $V^\prime$ is isomorphic to $\mathcal{G}$ if and only if
    \begin{align*}
        \left[\begin{array}{c|c|c}
            1& 0 & 0 \\
            \hline
            0 & R & 0 \\
            \hline 
            0 & 0 & 1
            \end{array}\right]
            V^\prime = V.
    \end{align*}
    holds for some $R \in \SO(2)$. Consequently, the statement follows from Lemma \ref{ntg::lemma:Pauli_spaces_for_qgraphs_with_one_edge_on_M2,psi_q}.
\end{proof}


In what follows, we discuss the quantum graphs $\mathcal{G}_\alpha^{(q,1B)}$ and $\mathcal{G}_{\alpha, \beta, \gamma}^{(q, 1C)}$ in greater detail. Let us start with the quantum graphs $\mathcal{G}_{\alpha}^{(q,1B)}$.

\begin{proposition}
\label{ntg::prop:qgraph_G_alpha_on_(M2,psi_q)_with_one_edge}
    Let $\alpha\in\C$.
    The quantum graph $\mathcal{G}_\alpha^{(q,1B)}$ has quantum edge space
    \begin{align*}
        S_\alpha^{(q,1B)} = \span\left\{ \begin{pmatrix}
            \alpha + q^{-2} & 0 \\
            0 & \alpha - 1
        \end{pmatrix} \right\}
    \end{align*}
    and quantum adjacency matrix
    \begin{align*}
        A_\alpha^{(q,1B)} = c^{-1}\begin{pmatrix}
            |\alpha+q^{-2}|^2 & 0 & 0 & 0 \\
            0 & (\alpha+q^{-2})\overline{(\alpha-1)} & 0 & 0 \\
            0 & 0 & \overline{(\alpha+q^{-2})}(\alpha-1) & 0 \\
            0 & 0 & 0 & |\alpha-1|^2
        \end{pmatrix},
    \end{align*}
    where 
    \begin{align*}
    c= \frac{1}{1+q^2}\left( q^2|\alpha+q^{-2}|^2 + |\alpha-1|^2 \right).
    \end{align*}
    Its spectrum is
    \begin{align*}
        c & \left\{ |\alpha+q^{-2}|^2, (\alpha+q^{-2})\overline{(\alpha-1)}, \overline{(\alpha+q^{-2})}(\alpha-1), |\alpha-1|^2\right\}.
    \end{align*}
    The quantum graph $\mathcal{G}_\alpha^{(q,1B)}$ is
    \begin{itemize}
        \item GNS-undirected if and only if it is KMS-undirected if and only if $\alpha \in \R$,
        \item never reflexive,
        \item loopfree if and only if $\alpha = 0$.
    \end{itemize}
\end{proposition}

\begin{proof}
    Let us skip the indices in the proof.
    The quantum edge space $S$ is obtained directly from the definition of the Pauli space. For
    \begin{align*}
        T := \begin{pmatrix}
            \alpha + q^{-2} & 0 \\
            0 & \alpha - 1
        \end{pmatrix}
    \end{align*}
    one observes 
    \begin{align*}
        \left(\|T\|_{\psi_q^{-1}}^{KMS}\right)^2 &= \langle T, T \rangle_{\psi_q^{-1}}^{KMS} 
            = \mathrm{Tr}\left( T^\ast \rho_q^{-\frac{1}{2}} T \rho_q^{-\frac{1}{2}} \right)\\
            &= \frac{1}{1+q^2} \left( q^2 |\alpha+q^{-2}|^2 + |\alpha-1|^2 \right) 
            = c.
    \end{align*}
    \comments{
    \begin{align*}
        \left(\|T\|_{\psi_q^{-1}}^{KMS}\right)^2 &= \langle T, T \rangle_{\psi_q^{-1}}^{KMS} \\
            &= \mathrm{Tr}\left( T^\ast \rho_q^{-\frac{1}{2}} T \rho_q^{-\frac{1}{2}} \right) \\
            &= \frac{1}{1+q^2} \mathrm{Tr}\left( T^\ast \begin{pmatrix}
                q & 0 \\
                0 & 1
            \end{pmatrix}
            T
            \begin{pmatrix}
                q & 0 \\
                0 & 1
            \end{pmatrix} \right) \\
            &= \frac{1}{1+q^2} \mathrm{Tr} \left( \begin{pmatrix}
                 q(\overline{\alpha+q^{-2}}) & 0\\
                 0 & \overline{\alpha-1}
            \end{pmatrix}
            \begin{pmatrix}
                q(\alpha+q^{-2}) & 0 \\
                0 & \alpha-1
            \end{pmatrix} \right) \\
            &= \frac{1}{1+q^2} \mathrm{Tr} \left( \begin{pmatrix}
                q^2|\alpha+q^{-2}|^2 & 0 \\
                0 & |\alpha-1|^2
            \end{pmatrix}
            \right) \\
            &= \frac{1}{1+q^2} \left( q^2 |\alpha+q^{-2}|^2 + |\alpha-1|^2 \right) \\
            &= c^{-1}.
    \end{align*}
    }
    Then, the quantum adjacency matrix can be computed with Proposition \ref{pre::prop:S->A} and Remark \ref{pre::rem:adjacency_matrix_A_as_Kronecker_product}
    \begin{align*}
    A = c^{-1}(T \otimes \overline{T}) = c^{-1} (\rho_q^{-\frac{1}{4}} T \rho_q^{\frac{1}{4}} \otimes \overline{\rho_q^{-\frac{1}{4}} T \rho_q^{\frac{1}{4}}} ).
    \end{align*}
    Its spectrum can be readily seen for it is a diagonal matrix.  Finally, recall from Definition \ref{pre::def:GNS-KMS-undirected_quantum_graphs} that $\mathcal{G}$ is GNS-undirected if and only if $S = S^\ast$ and $\rho_q S \rho_q^{-1} = S$. The first condition is equivalent to asking $\alpha \in \R$, and the latter is always satisfied for 
    \begin{align*}
        \rho_q T \rho_q^{-1} = T.
    \end{align*}
    In particular, the quantum graph is GNS-undirected if and only if it is KMS-undirected. 
    Further, the quantum graph $\mathcal{G}$ is never reflexive since $q^{-2}\neq-1$ for all $q\in(0,1)$ and thus, $\mathrm{I}_2\notin S$.
    Lastly, $\mathcal{G}$ is loopfree if and only if $\langle T, \mathrm{I}_2 \rangle_{\psi_q^{-1}}^{KMS} = 0$ which is equivalent to $\alpha = 0$.
\end{proof}

Next, we investigate the quantum graphs $\mathcal{G}_{\alpha,\beta,\gamma}^{(q,1C)}$ with $(\alpha,\beta,\gamma)\in J^{(1,qC)}$.

\begin{proposition}
    \label{ntg::prop:qgraph_G_alpha,beta,gamma_with_one_edge}
    Let $(\alpha,\beta,\gamma)\in J^{(1,qC)}$ and denote $\beta_+ := 1+\beta$, $\beta_- := 1 - \beta$, $\alpha_+ := \alpha+q^{-2}\gamma$, $\alpha_- := \alpha-\gamma$.
    The quantum graph $\mathcal{G}_{\alpha,\beta,\gamma}^{(q,1C)}$ has the quantum edge space
    \begin{align*}
        S_{\alpha,\beta,\gamma}^{(q,1C)} = \spann\left\{ \begin{pmatrix}
            \alpha + q^{-2} \gamma & 1 + \beta \\
            1-\beta & \alpha - \gamma
        \end{pmatrix} \right\}
    \end{align*}
    and quantum adjacency matrix 
    \begin{align*}
        A_{\alpha,\beta,\gamma}^{(q,1C)} = c^{-1} \begin{pmatrix}
            |\alpha_+|^2     & q^{\frac{1}{2}} \beta_+ \alpha_+            & q^{\frac{1}{2}}\beta_+ \overline{\alpha_+}              & q \beta_+^2            \\
            q^{-\frac{1}{2}} \beta_- \alpha_+ & \alpha_+ \overline{\alpha_-} & \beta_+ \beta_-                           & q^{\frac{1}{2}}\beta_+ \overline{\alpha_-}   \\
            q^{-\frac{1}{2}}\beta_- \overline{\alpha_+} & \beta_+ \beta_-                          & \alpha_- \overline{\alpha_+}   & q^{\frac{1}{2}}\beta_+ \alpha_-   \\
            q^{-1}\beta_-^2         & q^{-\frac{1}{2}}\beta_- \overline{\alpha_-}            & q^{-\frac{1}{2}} \beta_- \alpha_-              & |\alpha_-|^2        
        \end{pmatrix},
    \end{align*}
    where
    \begin{align*}
        c \vcentcolon= \frac{1}{1+q^2} \left( q^2 |\alpha_+|^2 + 2q (1+\beta^2) + |\alpha_-|^2 \right).
    \end{align*}
    Its spectrum is the set
    \begin{align*}
        c^{-1} \left\{ |s_+|^2, s_+ \overline{s_-}, \overline{s_+} s_-, |s_-|^2\right\},
    \end{align*}
    for
    \begin{align*}
        s_{\pm} := 
            \alpha + \frac{q^{-2}-1}{2} \gamma \pm \sqrt{\left(\frac{q^{-2}+1}{2}\right)^2\gamma^2 + 1-\beta^2}.
    \end{align*}
    The quantum graph $\mathcal{G}_{\alpha,\beta,\gamma}^{(q,1C)}$ is 
    \begin{itemize}
        \item not GNS-undirected,
        \item KMS-undirected if and only if $\beta = 0$ and $\alpha, \gamma \in \R$,
        \item never reflexive,
        \item loopfree if and only if $\alpha = 0$.
    \end{itemize}
\end{proposition}

\begin{proof}
    Let us omit the indices for the sake of brevity.
    The space $S$ is obtained directly from the Pauli space $V$. Let
    \begin{align*}
        T := \begin{pmatrix}
            \alpha + q^{-2}\gamma & 1 + \beta \\
            1-\beta & \alpha - \gamma
        \end{pmatrix}
        = \begin{pmatrix}
            \alpha_+ & \beta_+ \\
            \beta_- & \alpha_-
        \end{pmatrix},
    \end{align*}
    and observe $S = \spann \{T\}$.
    The norm of $T$ with respect to the KMS-inner product induced by $\psi_q^{-1}$ is given by 
    \begin{align*}
        \left(\|T\|_{\psi_q^{-1}}^{KMS}\right)^2 
        =\frac{1}{1+q^2} \left( q^2 |\alpha_+|^2 + 2q (1+\beta^2) + |\alpha_-|^2 \right)
        = c
    \end{align*}
    where the calculation is similar to the preceding lemma.
    \comments{
    \begin{align*}
        \left(\|T\|_{\psi_q^{-1}}^{KMS}\right)^2 &= \langle T, T \rangle_{\psi_q^{-1}}^{KMS} \\
            &= \mathrm{Tr}\left( T^\ast \rho_q^{-\frac{1}{2}} T \rho_q^{-\frac{1}{2}} \right) \\
            &= \frac{1}{1+q^2} \mathrm{Tr}\left( T^\ast \begin{pmatrix}
                q & 0 \\
                0 & 1
            \end{pmatrix}
            T
            \begin{pmatrix}
                q & 0 \\
                0 & 1
            \end{pmatrix} \right) \\
            &= \frac{1}{1+q^2} \mathrm{Tr} \left( \begin{pmatrix}
                 q \overline{\alpha_+} & \beta_- \\
                 q \beta_+ & \overline{\alpha_-}
            \end{pmatrix}
            \begin{pmatrix}
                q \alpha_+ & \beta_+ \\
                q \beta_- & \alpha_-
            \end{pmatrix} \right) \\
            &= \frac{1}{1+q^2} \mathrm{Tr} \left( \begin{pmatrix}
                q^2|\alpha_+|^2 + q \beta_-^2 & \ast \\
                \ast & q \beta_+^2 + |\alpha_-|^2
            \end{pmatrix}
            \right) \\
            &= \frac{1}{1+q^2} \left( q^2 |\alpha_+|^2 + 2q (1+\beta^2) + |\alpha_-|^2 \right) \\
            &\vcentcolon= c^{-1}.
    \end{align*}
    }
    Now, $\left\{ c^{-1/2} T \right\}$ is an orthonormal basis of $S$ with respect to the KMS-inner product induced by $\psi_q^{-1}$. Therefore, Proposition \ref{pre::prop:S->A} and Remark \ref{pre::rem:adjacency_matrix_A_as_Kronecker_product} yield
    \begin{align*}
        A&= c \left( \rho_q^{-\frac{1}{4}} T \rho_q^{\frac{1}{4}} \right) \otimes \overline{\left( \rho_q^{-\frac{1}{4}} T \rho_q^{\frac{1}{4}} \right)} 
            = c \left( 
                \begin{pmatrix}
                    \alpha_+ & q^{\frac{1}{2}} \beta_+ \\
                    q^{-\frac{1}{2}} \beta_- & \alpha_- 
                \end{pmatrix} 
                \otimes
                \begin{pmatrix}
                    \overline{\alpha_+} & q^{\frac{1}{2}} \beta_+ \\
                    q^{-\frac{1}{2}} \beta_- & \overline{\alpha_-} 
                \end{pmatrix}\right),
    \end{align*}
    and the statement follows by applying the Kronecker product. 
    The spectrum of 
    \begin{align*}
        \begin{pmatrix}
            \alpha_+ & q^{\frac{1}{2}} \beta_+ \\
                    q^{-\frac{1}{2}} \beta_- & \alpha_- 
        \end{pmatrix}
    \end{align*}
    consists of the two numbers
    \begin{align*}
        s_{\pm} := 
            \alpha + \frac{q^{-2}-1}{2} \gamma \pm \sqrt{\left(\frac{q^{-2}+1}{2}\right)^2\gamma^2 + 1-\beta^2}
    \end{align*}
    and the statement about the spectrum of $A_{\alpha, \beta, \gamma}$ follows immediately.
    
    Finally, recall from Definition \ref{pre::def:GNS-KMS-undirected_quantum_graphs} that $\mathcal{G}$ is GNS-undirected if and only if $S = S^*$ and $S = \rho_q S \rho_q^{-1}$. 
    Since $S = \spann\{T\}$, this condition is equivalent to 
    \begin{align*}
        \begin{pmatrix}
            \overline{\alpha_+} & \beta_- \\
            \beta_+  & \overline{\alpha_-}
        \end{pmatrix}
        = \lambda \begin{pmatrix}
            \alpha_+ & \beta_+ \\
            \beta_-  & \alpha_-
        \end{pmatrix}
        \text{ and }
        \begin{pmatrix}
            \alpha_+ & q^{-2} \beta_+ \\
            q^{2} \beta_-  & \alpha_-
        \end{pmatrix}
        = \mu \begin{pmatrix}
            \alpha_+ & \beta_+ \\
            \beta_-  & \alpha_-
        \end{pmatrix}
    \end{align*}
    for some $\lambda, \mu \in \C$. As $\beta_+, \beta_- \in [0,2]$ the first equation entails $\lambda\in [0, \infty)$, and thus the equation holds if and only if $\beta_+ = \beta_-$ and $\alpha_+, \alpha_- \in \R$ which is equivalent to $\alpha, \gamma \in \R$ and $\beta = 0$. Then, however, the second equality asks $q^2 = q^2 \beta_- = \mu \beta_- = \mu$ as well as $q^{-2} = \mu$. This is impossible for $q \in (0,1)$. Hence, $\mathcal{G}$ is not GNS-undirected for any admissible $\alpha, \beta, \gamma$ and $q$.
    However, the quantum graph is KMS-undirected if and only if $S = S^*$ which holds if and only if $\beta = 0$ and $\alpha, \gamma \in \R$. 
    Further, $\beta_+=\beta_-=0$ is not possible for any $\beta\in[-1,1]$ and thus we have $\mathrm{I}_2\notin S$. Hence, $\mathcal{G}$ is never reflexive.
    Lastly, $\mathcal{G}$ is loopfree if and only if $\langle T, \mathrm{I}_2 \rangle_{\psi_q^{-1}}^{KMS} = 0$. We choose the $q$-adjusted Pauli space so that this is equivalent to $\alpha = 0$.
\end{proof}

\begin{remark}
\label{ntg::rmk:KMS-undirected_qgraphs_with_one_edge_on_(M2,psi_q)}
    The quantum graph $\mathcal{G}_{\alpha, \beta, \gamma}^{(q,1C)}$ is loopfree and KMS-undirected if and only if $\beta = 0$, $\alpha = 0$ and $\gamma \in \R$. In fact, in this case it is $\gamma \geq 0$ since $(\alpha, \beta, \gamma) \in J^{(q, 1C)}$. Thus, the loopfree KMS-undirected quantum graphs on $(M_2, \psi_q)$ are actually indexed by a single non-negative number $\gamma \geq 0$ and given by the quantum adjacency matrix
    \begin{align*}
        A_{0,0,\gamma}^{(q,1)} &= c \begin{pmatrix}
            |\alpha_+|^2     & q^{\frac{1}{2}} \beta_+ \alpha_+            & q^{\frac{1}{2}}\beta_+ \overline{\alpha_+}              & q \beta_+^2            \\
            q^{-\frac{1}{2}} \beta_- \alpha_+ & \alpha_+ \overline{\alpha_-} & \beta_+ \beta_-                           & q^{\frac{1}{2}}\beta_+ \overline{\alpha_-}   \\
            q^{-\frac{1}{2}}\beta_- \overline{\alpha_+} & \beta_+ \beta_-                          & \alpha_- \overline{\alpha_+}   & q^{\frac{1}{2}}\beta_+ \alpha_-   \\
            q^{-1}\beta_-^2         & q^{-\frac{1}{2}}\beta_- \overline{\alpha_-}            & q^{-\frac{1}{2}} \beta_- \alpha_-              & |\alpha_-|^2
        \end{pmatrix} \\
        &= \frac{1+q^2}{(q^{-2}+1)\gamma^2+2q} \begin{pmatrix}
            q^{-4} \gamma^2 & q^{-\frac{3}{2}} \gamma & q^{-\frac{3}{2}} \gamma & q \\
            q^{-\frac{5}{2}} \gamma & - q^{-2} \gamma^2 & 1 & -q^{\frac{1}{2}} \gamma \\
            q^{-\frac{5}{2}} \gamma & 1 & - q^{-2} \gamma^2 & -q^{\frac{1}{2}} \gamma \\
            q^{-1} & -q^{\frac{1}{2}} \gamma & -q^{\frac{1}{2}} \gamma & \gamma^2
        \end{pmatrix}.
    \end{align*}
\end{remark}

\begin{remark}
    Let us discuss how the nontracial quantum graphs with one quantum edge fits into Matsuda's classification of GNS-undirected quantum graphs on $(M_2, \omega_q)$, where $\omega_q = \frac{1}{(q+q^{-1})^2} \psi_q$. In \cite[Theorem 3.5(2)]{matsuda_classification_2022} the author proves that there exists exactly one loopfree quantum graph with one quantum edge on $(M_2, \omega_q)$, and this quantum graph has the quantum adjacency matrix
    \begin{align*}
        A = \begin{pmatrix}
            q^{-2} & 0 & 0 & 0 \\
            0 & -1 & 0 & 0 \\
            0 & 0 & -1 & 0 \\
            0 & 0 & 0 & q^{2}
        \end{pmatrix}.
    \end{align*}
    Note that the original result is on reflexive (instead of loopfree) quantum graphs and that Matsuda uses another basis of $M_2$, but this can be easily adapted. Combining our previous observations, up to isomorphism the only GNS-undirected and loopfree quantum graph on $(M_2, \psi_q)$ is $\mathcal{G}_{\alpha}^{(q,1B)}$ with $\alpha = 0$. Indeed, an inspection of Lemma \ref{ntg::lemma:Pauli_spaces_for_qgraphs_with_one_edge_on_M2,psi_q} shows that there are no distinct quantum graphs on $(M_2, \psi_q)$ which are isomorphic to this quantum graph $\mathcal{G}_{0}^{(q,1B)}$. A detailed computation of the adjacency matrix $A_0^{(q,1B)}$ from Proposition \ref{ntg::prop:qgraph_G_alpha_on_(M2,psi_q)_with_one_edge} yields
    \begin{align*}
        A_0^{(q,1B)} = 
        \begin{pmatrix}
            q^{-2} & 0 & 0 & 0 \\
            0 & -1 & 0 & 0 \\
            0 & 0 & -1 & 0 \\
            0 & 0 & 0 & q^{2}
        \end{pmatrix}.
    \end{align*}
    This is exactly as expected. Indeed, one can verify that $A: M_2 \to M_2$ is a quantum adjacency matrix on $(M_2, \psi_q)$ in our framework exactly if it is a quantum adjacency matrix on $(M_2, \omega_q)$ in Matsuda's framework.
\end{remark}

\section{Non-Loopfree Nontracial Quantum Graphs with Two Quantum Edges}
\label{sec::n2e}

To complete the classificaton of nontracial quantum graphs it remains to investigate the nontracial non-loopfree quantum graphs with two quantum edges. By Remark \ref{iso::remark:complete_classification}, together with Theorem \ref{ntg::thm:classification_nontracial_qgraphs_with_one_edge} this yields the complete classification of nontracial quantum graphs on $M_2$.

\begin{lemma}
\label{ntg::lemma:2edges_plane_classification}
    Let $V \subset \C^4$ be a plane with $V \not \perp e_1$. 
    Then there are unique numbers $\beta \in [-1, 1]$ and $\alpha, \gamma, \delta \in \C$ such that 
    \begin{align*}
        \left[\begin{array}{c|c|c}
            1& 0 & 0 \\
            \hline
            0 & R & 0 \\
            \hline 
            0 & 0 & 1
            \end{array}\right]
        V = 
            \spann \left\{ \begin{pmatrix}
                0 \\ 1 \\ i \beta \\ \delta
            \end{pmatrix}, \begin{pmatrix}
                1 \\ i \alpha \beta - \gamma \overline{\delta} \\ \alpha \\ \gamma
            \end{pmatrix} \right\}
            =\vcentcolon V^{(q,2)}_{\alpha,\beta,\gamma,\delta}
    \end{align*}
    holds for some $R \in \SO(2)$ and
    \begin{align}
        \label{8::eq:J^(q,2)}\tag{$\ast$}
        \begin{aligned}
            \beta = \pm 1, \delta \neq 0 &\implies \delta \in (0, \infty), \\
            \beta = \pm 1, \delta = 0 &\implies \alpha \in [0, \infty), \\
            \beta \in (-1,1), \delta \neq 0 &\implies \mathrm{arg}(\delta) \in [0, \pi), \\
            \beta \in (-1,1), \delta = 0 &\implies \mathrm{arg}(\alpha) \in [0, \pi).
        \end{aligned}
    \end{align}
    Denote the index set 
    $J^{(q,2)}\vcentcolon=
        \{(\alpha,\beta,\gamma,\delta)\vcentcolon\text{\eqref{8::eq:J^(q,2)} is fulfilled}\}$.
\end{lemma}

\begin{proof}
    Up to scalar multiplication there is a unique non-vanishing vector $v \in V$ of the form
    \begin{align*}
        v = \begin{pmatrix}
            0 \\ v_1 \\ v_2 \\ v_3
        \end{pmatrix} \in V.
    \end{align*}
    Indeed, if there were two linearly independent vectors with vanishing first component, they would span the whole space $V$ -- in contradiction to $V \not \perp e_1$. Using Lemma \ref{lin::lemma:existence_of_index_beta_for_lines_in_C3} one finds a unique $\beta \in [-1,1]$ such that
    \begin{align*}
        \begin{pmatrix}
            v_1 \\ v_2
        \end{pmatrix}
        = 
        \lambda R \begin{pmatrix}
                1 \\ i \beta
        \end{pmatrix}
    \end{align*}
    holds for some $\lambda \in \C$ and $R \in \SO(2)$. Consequently, we have with $\delta := \lambda v_3$
    \begin{align*}
        \begin{pmatrix}
                0 \\ 1 \\ i \beta \\ \delta
            \end{pmatrix}
        = 
        \lambda \left[\begin{array}{c|c|c}
            1& 0 & 0 \\
            \hline
            0 & R & 0 \\
            \hline 
            0 & 0 & 1
            \end{array}\right]
            v.
    \end{align*}
    Let $w \in V$ be another vector that is orthogonal to $v$ and such that $V = \spann \{ v, w\}$. Then $w$ has non-vanishing first component since $v$ is not orthogonal to $e_1$ and without loss of generality we may assume that its first component is $w_1 = 1$. Using $v \perp w$ one readily checks that there are numbers $\alpha, \gamma \in \C$ such that
    \begin{align*}
        \begin{pmatrix}
                1 \\ i \alpha \beta - \gamma \overline{\delta} \\ \alpha \\ \gamma
            \end{pmatrix}
        =
        \left[\begin{array}{c|c|c}
            1& 0 & 0 \\
            \hline
            0 & R & 0 \\
            \hline 
            0 & 0 & 1
            \end{array}\right]
            w.
    \end{align*}
    It remains to investigate uniqueness of the numbers $\alpha, \beta, \gamma, \delta \in \C$. 
    For that, assume
    \begin{align*}
        \spann \left\{ \begin{pmatrix}
            0 \\ 1 \\ i \beta \\ \delta
        \end{pmatrix}, \begin{pmatrix}
            1 \\ i \alpha \beta - \gamma \overline{\delta} \\ \alpha \\ \gamma
        \end{pmatrix} \right\}
        = 
        \left[\begin{array}{c|c|c}
            1& 0 & 0 \\
            \hline
            0 & R & 0 \\
            \hline 
            0 & 0 & 1
            \end{array}\right]
            \spann \left\{ \begin{pmatrix}
                0 \\ 1 \\ i \beta^\prime \\ \delta^\prime
            \end{pmatrix}, \begin{pmatrix}
                1 \\ i \alpha^\prime \beta^\prime - \gamma^\prime \overline{\delta^\prime} \\ \alpha^\prime \\ \gamma^\prime
            \end{pmatrix} \right\}
    \end{align*}
    for some $R \in \SO(2)$ and $\beta^\prime \in [-1,1]$, $\alpha^\prime, \gamma^\prime, \delta^\prime \in \C$. Looking at the first components of the spanning vectors one readily checks
    \begin{align}
        \begin{pmatrix}
            0 \\ 1 \\ i \beta \\ \delta
        \end{pmatrix}
        &=
        \lambda' \left[\begin{array}{c|c|c}
            1& 0 & 0 \\
            \hline
            0 & R & 0 \\
            \hline 
            0 & 0 & 1
        \end{array}\right]
        \begin{pmatrix}
            0 \\ 1 \\ i \beta^\prime \\ \delta^\prime
        \end{pmatrix} 
    \end{align}
    for some $\lambda' \in \C$ and
    \begin{align}
    \label{ntg::eq:plane_lemma_case3}
        \begin{pmatrix}
            1 \\ i \alpha \beta - \gamma \overline{\delta} \\ \alpha \\ \gamma
        \end{pmatrix}
        &= 
        \left[\begin{array}{c|c|c}
            1& 0 & 0 \\
            \hline
            0 & R & 0 \\
            \hline 
            0 & 0 & 1
        \end{array}\right]
        \begin{pmatrix}
            1 \\ i \alpha^\prime \beta^\prime - \gamma^\prime \overline{\delta^\prime} \\ \alpha^\prime \\ \gamma^\prime
        \end{pmatrix}.
    \end{align}
    The first equation entails $\delta = \lambda' \delta^\prime$ as well as $\beta = \beta^\prime$ (due to Lemma \ref{lin::lemma:uniqueness_of_lambda_in_C2}), while the second entails $\gamma = \gamma^\prime$. 
    Thus, the numbers $\beta, \gamma$ are unique. 
    By Lemma \ref{lin::lemma:uniqueness_of_lambda_in_C2} the two equations hold if and only if -- additional to $\beta = \beta^\prime$, $\gamma = \gamma^\prime$ and $\delta = \lambda'\delta^\prime$ -- one of the following is true.
    \begin{enumerate}[label=(\alph*)]
        \item $\beta =\pm 1$ and $\displaystyle R = \begin{pmatrix}
            x & -\beta y \\
            \beta y & x
        \end{pmatrix}$
        for $x+ iy = \lambda' \in S^1$. By choosing $\lambda'$ appropriately we can obtain $\delta \in [0, \infty)$. If we had $\delta, \delta^\prime \in (0, \infty)$, then one checks immediately $\lambda' = 1$ and thus, $R=\mathrm{I}_2$ and $\alpha = \alpha^\prime$. Thus, in this case, it  follows that the numbers $\alpha, \beta, \gamma, \delta$ are unique if one asks $\delta \in [0, \infty)$. If, on the other hand, $\delta = 0$, then (\ref{ntg::eq:plane_lemma_case3}) is equivalent to
        \begin{align*}
            \alpha \begin{pmatrix}
                i \beta \\ 1
            \end{pmatrix}
            = \alpha^\prime 
            \begin{pmatrix}
                x & -\beta y \\
                \beta y & x
            \end{pmatrix}
            \begin{pmatrix}
                i \beta^\prime \\ 1
            \end{pmatrix}
        \end{align*}
        A simple calculation shows that this is true if and only if $\alpha = \lambda' \alpha^\prime$. 
        Thus, in that case one can choose $\alpha \in [0, \infty)$, and the numbers $\alpha, \beta, \gamma, \delta$ are unique if one asks $\alpha \in [0, \infty)$.
        \item $\beta \in (-1,1)$, $\lambda' = \pm 1$ and $R =\lambda^\prime \mathrm{I}_2$. In this case, the two equations entail $\alpha = \pm\alpha^\prime$ and $\delta = \pm\delta^\prime$.
    \end{enumerate}
    The latter two cases show that the pair $(\alpha, \delta)$ is unique up to a sign if $\beta \in (-1,1)$. Altogether, the statement follows.
\end{proof}

Using the previous lemma we arrive at the following classification of nontracial non-loopfree quantum graphs with two quantum edges. Recall the $q$-adjusted Pauli space from Definition \ref{pre::def:q_adjusted_Pauli_space}.

\begin{thm}
    \label{ntg::thm:classification_of_nontracial_qgraphs_with_two_edges}
    Let $\mathcal{G}_{\alpha, \beta, \gamma, \delta}^{(q,2)}$ be the quantum graph on $(M_2, \psi_q)$ given by the $q$-adjusted Pauli space $V_{\alpha, \beta, \gamma, \delta}^{(q,2)}$ where $q \in (0,1)$ .
    Every non-loopfree quantum graph $\mathcal{G}$ on $(M_2, \psi_q)$ with exactly two quantum edges is isomorphic to exactly one of the quantum graphs $\mathcal{G}_{\alpha, \beta, \gamma, \delta}^{(q,2)}$ where  $(\alpha,\beta, \gamma, \delta) \in J^{(q,2)}$.
\end{thm}

\begin{proof}
    Let $\mathcal{G}$ be a quantum graph on $(M_2, \psi_q)$ with $q$-adjusted Pauli space $V^{(q)} \subset \C^4$. As $\mathcal{G}$ is non-loopfree we have $V \not \perp e_1$ by Lemma \ref{pau::lemma:q_adjusted_Pauli_spaces_of_isomorphic_nontracial_qgraphs}. Thus, Lemma \ref{ntg::lemma:2edges_plane_classification} yields that $\mathcal{G}$ is isomorphic to one of the quantum graphs $\mathcal{G}_{\alpha, \beta, \gamma, \delta}^{(q,2)}$.
\end{proof}

Finally, let us investigate the quantum graphs $\mathcal{G}_{\alpha, \beta, \gamma, \delta}^{(q,2)}$.

\begin{proposition}
\label{ntg::prop:qgraph_G_alpha,beta,gamma,delta_with_two_edges}
    The quantum graph $\mathcal{G}_{\alpha, \beta, \gamma, \delta}^{(q,2)}$ has quantum edge space
    \begin{align*}
        S_{\alpha, \beta, \gamma, \delta}^{(q,2)} = \spann \left\{
            \begin{pmatrix}
                q^{-2} \delta & 1+\beta \\
                1-\beta & - \delta
            \end{pmatrix},
            \begin{pmatrix}
                1+q^{-2} \gamma    & i\alpha(\beta-1)-\gamma\overline{\delta} \\
                i\alpha(\beta+1)-\gamma\overline{\delta} & 1-\gamma
            \end{pmatrix}
        \right\}.
    \end{align*}
    It is
    \begin{itemize}
        \item not GNS-undirected,
        \item KMS-undirected if and only if $\beta = 0$ and $\alpha, \gamma, \delta \in \R$,
        \item reflexive if and only if $\alpha = \gamma = 0$,
        \item never loopfree. 
    \end{itemize}
\end{proposition}

\begin{proof}
    To simplify notation, we omit the indices in the proof.
    The first statement follows directly from the definition of the $q$-adjusted Pauli space of $\mathcal{G}$. Let
    \begin{align*}
        T_1 := \begin{pmatrix}
            q^{-2} \delta & 1+\beta \\
            1-\beta & - \delta
        \end{pmatrix}
        \quad \text{ and } \quad 
        T_2 := \begin{pmatrix}
            1+q^{-2} \gamma    & i\alpha(\beta-1)-\gamma\overline{\delta} \\
            i\alpha(\beta+1)-\gamma\overline{\delta} & 1-\gamma
        \end{pmatrix}.
    \end{align*}
    Recall that $\mathcal{G}$ is KMS-undirected if and only if $S^\ast = S$. This is equivalent to $T_1^\ast, T_2^\ast \in S$ which entails for one $T_1^\ast = \mu_1 T_1 + \nu_1 T_2$, respectively in the $q$-adjusted Pauli space
    \begin{align*}
        \begin{pmatrix}
            0 \\ 1 \\ -i \beta \\ \overline{\delta}
        \end{pmatrix}
        = \mu_1 \begin{pmatrix}
            0 \\ 1 \\ i \beta \\ \delta
        \end{pmatrix}
        + \nu_1 \begin{pmatrix}
            1 \\ i \alpha \beta - \gamma \overline{\delta} \\ \alpha \\ \gamma
        \end{pmatrix}
    \end{align*}
    for some $\mu_1, \nu_1 \in \C$. This holds if and only if $\mu_1=1$, $\nu_1=0$ and $\beta = 0$, $\delta \in \R$. Using this in the analogous equation for $T_2^\ast \in S$ one obtains the requirement
    $T_2^\ast = \mu_2 T_1 + \nu_2 T_2$
    which is equivalent to
    \begin{align*}
        \begin{pmatrix}
            1 \\ -\overline{\gamma}\delta \\ \overline{\alpha} \\ \overline{\gamma}
        \end{pmatrix}
        = \mu_2 \begin{pmatrix}
            0 \\ 1 \\ 0 \\ \delta
        \end{pmatrix}
        + \nu_2 \begin{pmatrix}
            1 \\ - \gamma \overline{\delta} \\ \alpha \\ \gamma
        \end{pmatrix}
    \end{align*}
    for some $\mu_2, \nu_2 \in \C$. This holds if and only if
    \begin{align*}
        \nu_2 = 1,\quad \alpha = \overline{\alpha},\quad
        -\overline{\gamma}\delta &= \mu_2 - \gamma \overline{\delta} 
        \quad\text{and}\quad
        \overline{\gamma} = \mu_2 \delta + \gamma.
    \end{align*}
    The latter two equations are equivalent to
    \begin{align*}
        2i\mathrm{Im}(\gamma \overline{\delta}) = \gamma \overline{\delta}  - \overline{\gamma}\delta = \mu_2 
        \quad\text{and}\quad
        - 2i\mathrm{Im}(\gamma) = \overline{\gamma} - \gamma = \mu_2 \delta.
    \end{align*}
    If $\delta = 0$, then this implies $\gamma \in \R$. Otherwise, using $\delta \in \R$ from above one has 
    \begin{align*}
        - \mu_2 \delta = 2i\mathrm{Im}(\gamma) = \frac{2i \mathrm{Im}(\gamma \overline{\delta})}{\delta} = \frac{\mu_2}{\delta}
    \end{align*}
    which implies $\mu_2 =0$ or $-\delta = \frac{1}{\delta}$. The latter is impossible for $\delta \in \R$, and thus we obtain $\mu_2 = 0$ and $\gamma\in\R$. Altogether, it follows that $S = S^\ast$ is equivalent to $\beta = 0$ and $\alpha, \gamma, \delta \in \R$.
    The quantum graph $\mathcal{G}$ is GNS-undirected if additionally $\rho_q S \rho_q^{-1} = S$. However, one checks
    \begin{align*}
        \rho_q T_1 \rho_q^{-1}
        = \rho_q \begin{pmatrix}
            q^{-2} \delta & 1 \\
            1 & - \delta
        \end{pmatrix}
        \rho_q^{-1}
        = \begin{pmatrix}
            q^{-2} \delta & q^{-2} \\
            q^2 & - \delta
        \end{pmatrix}
        \not \in S.
    \end{align*}
    Indeed, all matrices in $S$ must have coinciding numbers in the upper right and lower left matrix entry since that is true for both $T_1$ and $T_2$. However, as $q \in (0,1)$ this is not true for $\rho_q T_1 \rho_q^{-1}$. Thus, the quantum graph $\mathcal{G}$ is not GNS-undirected.

    Next, recall that $\mathcal{G}$ is reflexive if $\mathrm{I}_2 \in S$. This is the case if and only if there are $\mu, \nu \in \C$ with $\mathrm{I}_2 = \mu T_1 + \nu T_2$ which is equivalent to
    \begin{align*}
        \begin{pmatrix}
            1 \\ 0 \\ 0 \\ 0
        \end{pmatrix}
        = \mu \begin{pmatrix}
            0 \\ 1 \\ i \beta \\ \delta
        \end{pmatrix}
        + \nu \begin{pmatrix}
            1 \\ i \alpha \beta - \gamma \overline{\delta} \\ \alpha \\ \gamma
        \end{pmatrix},
    \end{align*}
    where we switch to the $q$-adjusted Pauli basis of $M_2$ in the latter line. Evidently, the equation is equivalent to $\mu = 0$, $\nu = 1$ and $\alpha = \gamma = 0$ for the two featuring vectors are orthogonal with respect to the usual scalar product on $\C^4$.
\end{proof}

\section{Complete Classification}
\label{sec::cc}

In this section we discuss how the previous results combine into a complete classification of directed quantum graphs on $M_2$. This classification is compactly summarized in Table \ref{tabular:tracial_QG} and Table \ref{tabular:non-tracial_QG}. There, we denote $\beta_+=1+\beta$ and $\beta_-\vcentcolon=1-\beta$. 

Evidently, a quantum graph $\mathcal{G}$ on a quantum set $(M_2,\psi)$ has zero to four quantum edges, see Definition \ref{def:quantumGraph}. In what follows, we discuss quantum graphs with zero, one, two, three and four quantum edges separately.

If $\mathcal{G}$ has exactly zero quantum edges then it must be the empty quantum graph from Example \ref{pre::example:empty_trivial_complete_qgraph}. On the other hand, if $\mathcal{G}$ has exactly four quantum edges then it must be the complete quantum graph presented in the same example.

Quantum graphs with exactly one quantum edge have been classified in Section \ref{sec::t1e} (on the tracial quantum set) and in Section \ref{sec::n1e} (on non-tracial quantum sets). Indeed, any (tracial) quantum graph on $(M_2, \tau)$ with exactly one quantum edge is isomorphic to exactly one of the quantum graphs $\mathcal{G}^{(1A)}$ from Proposition \ref{pre::example:empty_trivial_complete_qgraph} and $\mathcal{G}_{\alpha,\beta}^{(1B)}$ from Example \ref{pre::example:empty_trivial_complete_qgraph} and Proposition \ref{t1e::prop:properties_of_G_alpha,beta^(1B)}, respectively. 
On the other hand, (non-tracial) quantum graphs on some $(M_2, \psi_q)$ with $q \in (0,1)$ with exactly one quantum edge are isomorphic to exactly one of the quantum graphs $\mathcal{G}^{(q,1A)}$, $\mathcal{G}_{\alpha}^{(q,1B)}$ and $\mathcal{G}_{\alpha,\beta,\gamma}^{(q,1C)}$ from Example \ref{pre::example:empty_trivial_complete_qgraph}, Proposition \ref{ntg::prop:qgraph_G_alpha_on_(M2,psi_q)_with_one_edge} and Proposition \ref{ntg::prop:qgraph_G_alpha,beta,gamma_with_one_edge}, respectively.


Next, assume that $\mathcal{G}$ is a quantum graph with exactly two quantum edges. If $\mathcal{G}$ is loopfree then $\mathcal{G}$ is the loopfree complement of a loopfree quantum graph with exactly one quantum edge in the sense of Lemma \ref{pre::lemma:1-1-correspondence_between_qgraphs_with_k_and_n2-k(-1)_edges}. 
So, let us assume that $\mathcal{G}$ is non-loopfree. 
We discussed the tracial case in Section \ref{sec::t2e} where we obtained that $\mathcal{G}$ is isomorphic to exactly one of the quantum graphs $\mathcal{G}_{\beta,\gamma,\delta}^{(2)}$ from Proposition \ref{t2e::prop:properties_of_G_alpha,beta,delta^(2)}. If $\mathcal{G}$ is non-tracial then it is isomorphic to exactly one of the quantum graphs $\mathcal{G}_{\alpha,\beta,\gamma,\delta}^{(q,2)}$ from Proposition \ref{ntg::prop:qgraph_G_alpha,beta,gamma,delta_with_two_edges}.


Finally, the quantum graphs with exactly three quantum edges are obtained as complements of the quantum graphs $\mathcal{G}^{(1A)}$, $\mathcal{G}_{\alpha,\beta}^{(1B)}$. $\mathcal{G}^{(q,1A)}$, $\mathcal{G}_{\alpha}^{(q,1B)}$ and $\mathcal{G}_{\alpha,\beta,\gamma}^{(q,1C)}$ with exactly one quantum edge by Lemma \ref{pre::lemma:1-1-correspondence_between_qgraphs_with_k_and_n2-k(-1)_edges}.

\comments{
The complete classification is summarized in the Table \ref{tabular:tracial_QG} and Table \ref{tabular:non-tracial_QG}. Table \ref{tabular:tracial_QG} lists all (tracial) quantum graphs on $(M_2,\tau)$ up to isomorphism where Table \ref{tabular:non-tracial_QG} lists all (nontracial) quantum graphs on $(M_2,\psi_q)$ for $q\in(0,1)$ up to isomorphism. 
We denote $\beta_+=1+\beta$ and $\beta_-\vcentcolon=1-\beta$. 
}

\newpage
\begin{table}[h]
\centering
\footnotesize
\begin{tabu}{c|[2pt]c|c|c}
&\\     & \textbf{quantum edge space} & \textbf{conditions on the variables} & \textbf{see} \\&&&\\
\tabucline[2pt]{-}
&&&\\
     \underline{\textbf{0 edges}}& $\{0\}$ & -- & \ref{pre::example:empty_trivial_complete_qgraph}\\
&&&\\
\hline
&&\\
    \underline{\textbf{1 edge}} &&\\
&&&\\
    $\mathcal{G}^{(1A)}$&$S^{(1A)}=\spann\bigg\{\begin{pmatrix}1&0\\0&1\end{pmatrix}\bigg\}$&--&\ref{pre::example:empty_trivial_complete_qgraph}\\
&&&\\
\hdashline
&&&\\
    $\mathcal{G}_{\alpha,\beta}^{(1B)}$&$S_{\alpha,\beta}^{(1B)}=\spann\bigg\{\begin{pmatrix}\alpha&\beta_+\\\beta_-&\alpha\end{pmatrix}\bigg\}$&
    $\begin{array}{r@{}l}
   \beta = 1 &\Rightarrow \alpha \in [0, \infty), \\
   \beta \in [0,1) &\Rightarrow \arg(\alpha) \in [0, \pi)
\end{array}$&\ref{t1e::prop:properties_of_G_alpha,beta^(1B)}\\
&&&\\
\hline
&&\\
    \underline{\textbf{2 edges}}&&\\
&&&\\
    $\mathcal{G}_{\beta,\gamma,\delta}^{(2)}$&
    $\begin{array}{l}S_{\beta,\gamma,\delta}^{(2)}\\
    =\spann\bigg\{\begin{pmatrix}
                0&\beta_+\\
                \beta_-&0
            \end{pmatrix},\\
            \begin{pmatrix}
                1+\delta&-i\gamma\beta_-\\
                i\gamma\beta_+&1-\delta
            \end{pmatrix}\bigg\}
            \end{array}$
    &$\begin{array}{r@{}l}
    \beta=0,\frac{\gamma}{\delta}\neq\pm i &\Rightarrow 
    \begin{cases}\arg(\gamma)\in[0,\pi),\\ 
    \arg(\delta)=\arg(\gamma)+\pi,\\
    |\delta|\leq|\gamma|,
    \end{cases}\\
    \beta=0,\frac{\gamma}{\delta}=\pm i &\Rightarrow 
    \begin{cases}\gamma\in[0,\infty),\\
    \arg(\delta)=\arg(\gamma)+\pi,\\
    |\delta|\leq|\gamma|,
    \end{cases}\\
    \beta\in(0,1) &\Rightarrow \arg(\gamma)\in[0,\pi),\\
    \beta=1 &\Rightarrow \gamma\in[0,\infty)
\end{array}$&\ref{t2e::prop:properties_of_G_alpha,beta,delta^(2)}\\
&&&\\
\hdashline
&&&\\
    $(\mathcal{G}_{0,\beta}^{(1B)})'$&$(S_{0,\beta}^{(1B)}\oplus\C\mathrm{I}_2)^\perp$&$\beta\in[0,1]$&\ref{pre::lemma:1-1-correspondence_between_qgraphs_with_k_and_n2-k(-1)_edges}, \ref{t1e::prop:properties_of_G_alpha,beta^(1B)}\\
&&&\\
\hline
&&\\
    \underline{\textbf{3 edges}}&&\\
&&&\\
    $(\mathcal{G}^{(1A)})^\perp$&$(S^{(1A)})^\perp$&--&\ref{pre::example:empty_trivial_complete_qgraph}, \ref{pre::lemma:1-1-correspondence_between_qgraphs_with_k_and_n2-k(-1)_edges}\\
&&&\\
\hdashline
&&&\\
    $(\mathcal{G}_{\alpha,\beta}^{(1B)})^\perp$&$(S_{\alpha,\beta}^{1B})^\perp$&see above&\ref{pre::lemma:1-1-correspondence_between_qgraphs_with_k_and_n2-k(-1)_edges}, \ref{t1e::prop:properties_of_G_alpha,beta^(1B)}\\
&&&\\
\hline
&&\\
    \underline{\textbf{4 edges}}&$M_2$& --&\ref{pre::example:empty_trivial_complete_qgraph}
\\&&&
\end{tabu}
\caption{(Tracial) Quantum Graphs on $(M_2,\tau)$.}
\label{tabular:tracial_QG}
\end{table}
\newpage

\begin{table}[h]
\centering
\footnotesize
\begin{tabu}{c|[2pt]c|c|c}
&&\\     & \textbf{quantum edge space} & \textbf{conditions on the variables} & \textbf{see} \\&&\\
\tabucline[2pt]{-}
&&&\\
     \underline{\textbf{0 edges}}& $\{0\}$ & -- & \ref{pre::example:empty_trivial_complete_qgraph}\\
&&&\\
\hline
&&\\
    \underline{\textbf{1 edge}} &&\\
&&&\\
    $\mathcal{G}^{(q,1A)}$&$S^{(q,1A)}=\spann\bigg\{\begin{pmatrix}1&0\\0&1\end{pmatrix}\bigg\}$&--&\ref{pre::example:empty_trivial_complete_qgraph}\\
&&&\\
\hdashline
&&&\\
    $\mathcal{G}_\alpha^{(q,1B)}$&$S_\alpha^{(q,1B)}=\spann\bigg\{\begin{pmatrix}\alpha+q^{-2}&0\\0&\alpha-1\end{pmatrix}\bigg\}$&$\alpha\in\C$&\ref{ntg::prop:qgraph_G_alpha_on_(M2,psi_q)_with_one_edge}\\
&&&\\
\hdashline
&&&\\
    $\mathcal{G}_{\alpha,\beta,\gamma}^{(q,1C)}$&$S_{\alpha,\beta,\gamma}^{(q,1C)}=\spann\bigg\{\begin{pmatrix}\alpha + q^{-2} \gamma & \beta_+ \\
            \beta_- & \alpha - \gamma\end{pmatrix}\bigg\}$&
    $\begin{array}{r@{}l}
   \beta = \pm 1, \alpha \neq 0 &\Rightarrow \alpha \in (0, \infty), \\
                \beta = \pm 1, \alpha = 0 &\Rightarrow \gamma \in [0, \infty), \\
                \beta \in (-1,1), \alpha \neq 0 &\Rightarrow \mathrm{arg}(\alpha) \in [0, \pi), \\
                \beta \in (-1,1), \alpha = 0 &\Rightarrow \mathrm{arg}(\gamma) \in [0, \pi)
\end{array}$&\ref{ntg::prop:qgraph_G_alpha,beta,gamma_with_one_edge}\\
&&&\\
\hline
&&\\
    \underline{\textbf{2 edges}}&&\\
&&&\\
    $\mathcal{G}_{\alpha,\beta,\gamma,\delta}^{(q,2)}$&
    $\begin{array}{l}
    S_{\alpha,\beta,\gamma,\delta}^{(q,2)}=\spann\bigg\{
    \begin{pmatrix}
                q^{-2} \delta & \beta_+ \\
                \beta_- & - \delta
            \end{pmatrix},\\
            \begin{pmatrix}
                1+q^{-2} \gamma    & -i\alpha\beta_--\gamma\overline{\delta} \\
                i\alpha\beta_+-\gamma\overline{\delta} & 1-\gamma
            \end{pmatrix}\bigg\}\end{array}$
    &$\begin{array}{r@{}l}
    \beta = \pm 1, \delta \neq 0 &\Rightarrow \delta \in (0, \infty), \\
            \beta = \pm 1, \delta = 0 &\Rightarrow \alpha \in [0, \infty), \\
            \beta \in (-1,1), \delta \neq 0 &\Rightarrow \mathrm{arg}(\delta) \in [0, \pi), \\
            \beta \in (-1,1), \delta = 0 &\Rightarrow \mathrm{arg}(\alpha) \in [0, \pi)
\end{array}$&\ref{ntg::prop:qgraph_G_alpha,beta,gamma,delta_with_two_edges}\\
&&&\\
\hdashline
&&&\\
    $(\mathcal{G}_{0}^{(q,1B)})'$&$(S_{0}^{(q,1B)}\oplus\C\mathrm{I}_2)^\perp$&--&\ref{pre::lemma:1-1-correspondence_between_qgraphs_with_k_and_n2-k(-1)_edges}, \ref{ntg::prop:qgraph_G_alpha_on_(M2,psi_q)_with_one_edge}\\
&&&\\
\hdashline
&&&\\
    $(\mathcal{G}_{0,\beta,\gamma}^{(q,1C)})'$&$(S_{0,\beta,\gamma}^{(q,1C)}\oplus\C\mathrm{I}_2)^\perp$&see above&\ref{pre::lemma:1-1-correspondence_between_qgraphs_with_k_and_n2-k(-1)_edges}, \ref{ntg::prop:qgraph_G_alpha,beta,gamma_with_one_edge}\\
&&&\\
\hline
&&\\
    \underline{\textbf{3 edges}}&&\\
&&&\\
    $(\mathcal{G}^{(q,1A)})^\perp$&$(S^{(q,1A)})^\perp$&--&\ref{pre::example:empty_trivial_complete_qgraph}, \ref{pre::lemma:1-1-correspondence_between_qgraphs_with_k_and_n2-k(-1)_edges}\\
&&&\\
\hdashline
&&&\\
    $(\mathcal{G}_{\alpha}^{(q,1B)})^\perp$&$(S_{\alpha}^{(q,1B)})^\perp$&see above&\ref{pre::lemma:1-1-correspondence_between_qgraphs_with_k_and_n2-k(-1)_edges}, \ref{ntg::prop:qgraph_G_alpha_on_(M2,psi_q)_with_one_edge}\\
    &&\\
\hdashline
&&&\\
    $(\mathcal{G}_{\alpha,\beta,\gamma}^{(q,1C)})^\perp$&$(S_{\alpha,\beta,\gamma}^{(q,1C)})^\perp$&see above&\ref{pre::lemma:1-1-correspondence_between_qgraphs_with_k_and_n2-k(-1)_edges}, \ref{ntg::prop:qgraph_G_alpha,beta,gamma_with_one_edge}\\
&&&\\
\hline
&&\\
    \underline{\textbf{4 edges}}&$M_2$& -- & \ref{pre::example:empty_trivial_complete_qgraph}
\\&&&
\end{tabu}
\caption{(Nontracial) Quantum Graphs on $(M_2,\psi_q)$ for $q\in(0,1)$.}
\label{tabular:non-tracial_QG}
\end{table}
\newpage


\printbibliography



\end{document}